\documentclass[final]{siamonline220329}

\usepackage{amsmath,amssymb}
\usepackage{thm-restate}
\usepackage{fullpage}
\usepackage{tikz}
\usepackage{wrapfig}
\usepackage{caption}
\usepackage{subcaption}
\usepackage{enumitem}
\usepackage{booktabs}
\usepackage[percent]{overpic}
\usepackage{multirow}

\usepackage{showlabels}

\usetikzlibrary{shapes,matrix,decorations.pathreplacing}

\setlist[enumerate]{leftmargin=.5in}
\setlist[itemize]{leftmargin=.5in}

\newcommand{\CC}{\ensuremath{\mathbb{C}}}
\newcommand{\RR}{\ensuremath{\mathbb{R}}}
\newcommand{\QQ}{\mathbb{Q}}
\newcommand{\ZZ}{\ensuremath{\mathbb{Z}}}
\newcommand{\imag}{\ensuremath{\mathfrak{i}}}
\newcommand{\init}{\ensuremath{\operatorname{init}}}
\newcommand{\newt}{\ensuremath{\operatorname{Newt}}}
\newcommand{\supp}{\ensuremath{\operatorname{supp}}}
\newcommand{\radjp}{\ensuremath{\check{\nabla}}}
\newcommand{\rowv}[1]{\ensuremath{\boldsymbol{#1}}}
\newcommand{\colv}[1]{\ensuremath{\vec{#1}}}
\newcommand{\dig}[1]{\ensuremath{\vec{#1}}}
\newcommand{\rinc}{\ensuremath{\check{Q}}}
\newcommand{\edges}{\ensuremath{\mathcal{E}}}
\newcommand{\nodes}{\ensuremath{\mathcal{V}}}
\newcommand{\neighbors}{\ensuremath{\mathcal{N}}}
\newcommand{\facets}{\mathcal{F}}
\newcommand{\inner}[2]{ \left\langle \, #1 \,,\, #2 \, \right\rangle }
\newcommand{\sinner}[2]{ \langle \, #1 \,,\, #2 \, \rangle }

\newcommand{\term}[1]{\textbf{#1}}
\newcommand{\subdiv}{\Sigma}
\newcommand{\seccone}{\mathcal{C}}
\newcommand{\aks}{\colv{f}}
\newcommand{\ake}{f}
\newcommand{\rake}{f^*}
\newcommand{\raks}{\colv{f}^*}
\newcommand{\im}{{\mathbf{i}}}

\newcommand{\badcoupling}{\mathcal{K}^\circ}

\DeclareMathOperator{\conv}{conv}
\DeclareMathOperator{\nvol}{Vol}
\DeclareMathOperator{\mvol}{MV}

\newsiamthm{question}{Question}
\crefname{question}{Question}{Questions}

\newsiamremark{example}{Example}
\newsiamremark{remark}{Remark}

\title{
    On the typical and atypical solutions to the Kuramoto equations\thanks{%
        Submitted to the editors DATE.
        \funding{%
            TC and EK are supported by a grant from
            the Auburn University at Montgomery Research Grant-in-Aid Program
            and the National Science Foundation under Grant No. 1923099.
            TC and JL are supported by
            the National Science Foundation under
            Grant No. 2318837.
            EK is also supported by the Undergraduate Research Experience program
            funded by the Department of Mathematics at Auburn University at Montgomery.
        }
    }
}

\author{
    Tianran Chen\thanks{%
        Department of Mathematics,
        Auburn University at Montgomery, Montgomery, AL
        (\email{ti@nranchen.org})
    } \and
    Evgeniia Korchevskaia\thanks{%
        School of Mathematics, Georgia Institute of Technology
    } \and
    Julia Lindberg\thanks{%
   Department of Mathematics, University of Texas at Austin,  
    Austin, TX
    (\email{julia.lindberg@math.utexas.edu})}
}

\begin{document}

\maketitle

\begin{abstract}
    The Kuramoto model is a dynamical system that
    models the interaction of coupled oscillators.
    There has been much work to effectively bound the number of equilibria
    to the Kuramoto model for a given network.
    By formulating the Kuramoto equations as a system of algebraic equations,
    we first relate the complex root count of the Kuramoto equations to the
    combinatorics of the underlying network
    by showing that the complex root count is generically equal to
    the normalized volume of the corresponding adjacency polytope of the network. 
    We then give explicit algebraic conditions under which this bound is strict and show that there are
 networks where the Kuramoto equations have infinitely many equilibria.
\end{abstract}


\section{Introduction}

The Kuramoto model \cite{Kuramoto1975Self} is a mathematical model
that describes the dynamics on networks of oscillators.
It has applications in neuroscience, biology, chemistry and power systems  
\cite{generative2010breakspear,synchronization2010dorfler,GUO2021106804,RODRIGUES20161}.
Despite its simplicity, it exhibits interesting emergent behaviors.
Of interest is the phenomenon of frequency synchronization
which is when the oscillators synchronize to a common frequency.
    Frequency synchronizations correspond to solutions of the 
    \emph{Kuramoto equations}
    \[
        \overline{w} =
        w_i - \sum_{j = 0}^n k_{ij} \sin(\theta_i - \theta_j) 
        \quad\text{for } i = 0,\dots,n, 
    \]
    in the unknowns $\theta_0,\ldots,\theta_n$.
    Here, $\overline{w}, w_0,\ldots,w_n,k_{ij}$ are network parameters.
This paper aims to understand the structure of these solutions
in ``typical'' and ``atypical'' networks.

Earlier work focused on the statistical analysis of
infinite networks \cite{Kuramoto1975Self} but more
recently, tools from differential and algebraic geometry
have enabled analysis of synchronizations on finite networks.
For a finite network, knowing the total number of synchronization configurations
is fundamental to understanding this model.
From a computational perspective,
this knowledge also plays a critical role in
developing numerical methods for finding synchronization configurations.
For instance, this number serves as a stopping criterion
for \emph{monodromy} algorithms
\cite{lindberg2020exploiting}
and allows for the development of specialized
\emph{homotopy} algorithms \cite{Chen2019Directed,ChenDavis2022Toric}
for finding all synchronization configurations.

In 1982, Ballieul and Byrnes introduced root counting techniques
from algebraic geometry to this field and showed that
a Kuramoto network of $N$ oscillators 
has at most $\binom{2N-2}{N-1}$ synchronization configurations 
\cite{BaillieulByrnes1982Geometric}.
It coincides with the bound on the root count
for the closely related load-flow equations
discovered by Li, Sauer, and Yorke \cite{LiSauerYorke1987Numerical}.
Algebraic geometers will recognize this bound
as the bi-homogeneous B\'ezout bound
for an algebraic version of the Kuramoto equations.
This upper bound can be reached when the network is complete
and complex roots 
are counted.
However, for sparse networks, 
the root count (even counting complex roots)
can be significantly lower than this upper bound \cite{GuoSalam1990Determining,MolzahnMehtaNiemerg2016Toward},
demonstrating the need for a network-dependent root count.

Guo and Salam initiated one of the first algebraic analyses
on such sparsity-dependent root counts \cite{GuoSalam1990Determining}.
Molzahn, Mehta, and Niemerg provided computational evidence
for the connection between this root count
and network topology \cite{MolzahnMehtaNiemerg2016Toward}.
In the special case of rank-one coupling coefficients,
Coss, Hauenstein, Hong and Molzahn
proved this complex root count to be $2^N - 2$,
which is also an asymptotically sharp bound on the real root count
\cite{coss2018locating}.
Chen, Davis and Mehta gave a sharp bound on the complex root count
for cycle networks of $N \binom{N-1}{ \lfloor (N-1) / 2 \rfloor }$ \cite{ChenDavisMehta2018Counting},
which is asymptotically smaller than the bi-homogeneous B\'ezout bound
discovered by Baillieul and Byrnes,
proving that sparse networks have significantly fewer synchronization configurations.
Interestingly, Lindberg, Zachariah, Boston, and Lesieutre showed that
this bound is attainable by real roots
\cite{LindbergZachariahBostonLesieutre2022Distribution}.
This bound is an instance of the 
``adjacency polytope bound'' \cite{Chen2019Unmixing},
which motivates the following.

\begin{question}\label{q1}
    For generic choices of network parameters,
    does the complex root count for the algebraic Kuramoto equations
    reach the adjacency polytope bound for all networks?
\end{question}
In the first part of this paper,
we provide a positive answer to this question
and thus establish the generic root count for the Kuramoto equations
derived from a graph $G$
to be the normalized volume of the adjacency polytope of $G$.
As a corollary, we show that the Kuramoto equations are Bernshtein-general.
We note that this result is similar to recent work in \cite{breiding2022the}
which shows that the number of (approximate) complex solutions to the
Duffing equations is generically the volume of the \emph{Oscillator polytope}.
We also extend this generic root count result to
variations of the Kuramoto equations,
including a special case of the power flow equations
from electric engineering.

While the complex root count for the algebraic Kuramoto equations
is generically constant and finite,
there may be network parameters
    that produce different root counts.
    First, we focus on the role played by the coupling coefficients
    in such exceptional situations.

\begin{question}\label{q2}
    What are conditions on the coupling coefficients under which
    the algebraic Kuramoto equations are \emph{not} Bernshtein-general?
\end{question}

In the second part of this paper,
we provide an explicit, combinatorial description
for such exceptional coupling coefficients.
In particular, we show that the set of exceptional coupling coefficients
can be characterized by ``balanced subnetworks''.

For the Kuramoto equations with exceptional coupling coefficients,
there are two possibilities.
Either all complex solutions remain isolated
but the total number drops below the generic root count,
or non-isolated solution components appear.
In the second case, there are infinitely many solutions,
forming curves, surfaces, or geometric structures of even higher dimension.

Ashwin, Bick, and Burylko analyzed
non-isolated solutions in complete networks of identical oscillators
\cite{AshwinBickBurylko2016Identical}.
A concrete example of a network of four identical oscillators
with uniform coupling coefficients was described in \cite[Example 2.1]{coss2018locating}.
Non-isolated solutions for cycle networks
was discovered by Lindberg, Zachariah, Boston and Lesieutre 
\cite{LindbergZachariahBostonLesieutre2022Distribution}.
Recent work by Sclosa shows that for every $d \geq 1$
there is a Kuramoto network whose stable equilibria
form a manifold of dimension $d$ \cite{sclosa2022kuramoto},
which, arguably, shows that the study of non-isolated solutions
deserves more serious attention.
This is further supported by the recent paper of
Harrington, Schenck, and Stillman \cite{harrington2023kuramoto},
which shows that
for any 2-connected graph that contains a 3-let,
there are parameters for which the Kuramoto equations
have a non-isolated solution set.
It is within this context, that the third part of this paper
aims to provide some explicit and constructive answers to the following.

\begin{question}\label{q3}
    What are the conditions on the network parameters under which
    the Kuramoto equations have infinitely many solutions?
\end{question}

The rest of this paper is structured as follows.
\Cref{sec: prelims} reviews concepts and notation
that will be used.
We then consider the generic root count of the Kuramoto equations in \Cref{sec: generic root count}
and provide a positive answer to \Cref{q1}.
Next, we turn our attention to non-generic coupling coefficients
in \Cref{sec: explicit genericity conditions}
and answer \Cref{q2} with a combinatorial description
of the exceptional coupling coefficients.
Using this description, in \Cref{sec: +dimensional}
we answer \Cref{q3} by identifying network parameters
where the Kuramoto equations have non-isolated solutions
and we construct explicit parameterizations for these solutions.
Finally, we conclude with a few remarks in \Cref{sec: conclusion}.

\section{Notation and preliminaries}\label{sec: prelims}

Column vectors, representing points of a lattice $L \cong \ZZ^n$,
are denoted by lowercase letters with an arrowhead, e.g., $\colv{a}$.
We use boldface letters, e.g., $\rowv{x}$,
for points in $\CC^n$, $\RR^n$, or $L^\vee \cong \ZZ^n$,
and they are written as row vectors.
For $\rowv{x} = (x_1, \dots, x_n)$
and $\colv{a} = (a_1,\dots,a_n)^\top \in \ZZ^n$,
$\rowv{x}^{\colv{a}} = x_1^{a_1} \, \cdots \, x_n^{a_n}$
is a \emph{Laurent monomial}
with the convention that $x_i^0 = 1$, for any $x_i$.
A linear combination of such monomials
$f = \sum_{\colv{a} \in S} \rowv{x}^{\colv{a}}$
is called a \emph{Laurent polynomial},
and its \emph{support} and \emph{Newton polytope} are denoted
$\supp(f) = S$ and $\newt(f) = \conv(S)$, respectively.
With respect to a vector $\rowv{v}$,
the \emph{initial form} of $f$ is
$\init_{\rowv{v}}(f)(\rowv{x}) := \sum_{\colv{a} \in (S)_{\rowv{v}}} c_{\colv{a}} \, \rowv{x}^{\colv{a}}$, 
where $(S)_{\rowv{v}}$ is the subset of $S$ on which
the linear functional $\inner{ \rowv{v} }{ \bullet }$ is minimized.
For an integer matrix $A = [ \colv{a}_1 \; \cdots \; \colv{a}_m]$,
$ \rowv{x}^A := ( \rowv{x}^{\colv{a}_1}, \dots, \rowv{x}^{\colv{a}_m} )$
defines a function on
the \emph{algebraic torus} $(\CC^*)^n = (\CC \setminus \{0\})^n$
whose group structure is given by 
$(x_1,\dots,x_n) \circ (y_1,\dots,y_n) := (x_1 y_1, \dots, x_n y_n)$.
We fix $G$ to be a connected graph
with vertex set $\nodes(G) = \{0,1,\dots,n\}$
and edge set $\edges(G)$.
For nodes $i$ and $j$ in $\nodes(G)$, $i \sim j$ indicates their adjacency.
Arrowheads will be used to distinguish digraphs from graphs,
e.g., $\dig{G}$ represents a digraph and $G$ an undirected graph.
\Cref{app: notations} has the full list of notation.

\subsection{The Kuramoto model}

A network of coupled oscillators can be naively thought of as
a swarm of points on the complex plane
pulling on one another with varying force
while circling around the origin.
It can be used to model a wide variety of
seemingly unrelated phenomena ranging from
the firing of neurons and the
rhythmic contractions of heart cells,
to the oscillations of concentrations of chemical compounds in a mixture.
In this paper, such a network is represented by
a connected graph $G$ whose nodes and edges
represent the oscillators and their connections, respectively.
Each oscillator has a natural frequency $w_i$
and along the edges in $G$, nonzero constants $K = \{ k_{ij} \}$
with $k_{ij} = k_{ji}$ quantify the coupling strength
between oscillators $i$ and $j$.
The data structure $(G,K,\colv{w})$ encoding this model
will simply be called a \term{network}.
The Kuramoto model describes the nonlinear interactions
among the oscillators by the differential equations
\begin{equation}\label{equ: kuramoto ode}
    \frac{d \theta_i}{dt} =
    w_i - \sum_{j \sim i} k_{ij} \sin(\theta_i - \theta_j) 
    \quad\text{for } i = 0,\dots,n, 
\end{equation}
where 
$\theta_i$ is the phase angle of the $i$-th oscillator \cite{Kuramoto1975Self}.
\emph{Frequency synchronization configurations}
are defined to be values of $(\theta_0,\dots,\theta_n)$ 
at which $\frac{d\theta_i}{dt}$ equals
$\overline{w} = \frac{1}{n} \sum_{i=0}^n w_i$.
By adopting a rotational frame,
we can assume $\theta_0 = 0$.
Since $k_{ij} = k_{ji}$,
we can also eliminate one equation.
Therefore, frequency synchronization configurations
are zeroes to the system of $n$ transcendental functions:
\begin{equation}\label{equ: kuramoto sin}
    (w_i - \overline{w}) - \sum_{j \sim i} k_{ij} \sin(\theta_i - \theta_j) 
    \quad\text{for } i = 1,\dots,n.
\end{equation}
The problem of counting synchronization configurations
is therefore a root counting problem.

\subsection{Algebraic Kuramoto equations}

To leverage the power of algebraic geometry,
the above transcendental system can be reformulated into an algebraic system
via the change of variables $x_i = e^{\imag \theta_i}$.
Then $\sin(\theta_i - \theta_j) = \frac{1}{2 \imag}(\frac{x_i}{x_j} - \frac{x_j}{x_i})$,
and~\eqref{equ: kuramoto sin} becomes
\begin{equation}\label{equ: algebraic kuramoto}
    \ake_{G,i}(x_1,\dots,x_n) = 
    \overline{w}_i - \sum_{j \sim i} 
        a_{ij}
        \left(
            \frac{x_i}{x_j} - \frac{x_j}{x_i}
        \right)
    \quad \text{for } i = 1,\dots,n,
\end{equation}
where $a_{ji} = a_{ij} = \frac{k_{ij}}{2\imag}$,
 $\overline{w}_i = w_i - \overline{w}$, and $x_0 = 1$.
The Laurent polynomial system
$\aks_G = (\ake_{G,1},\dots,\ake_{G,n})^\top$
will be called the \term{algebraic Kuramoto system},
and it captures all synchronization configurations in the sense that 
the real zeros to~\eqref{equ: kuramoto sin} correspond to the
complex zeros of~\eqref{equ: algebraic kuramoto} 
on the real torus $(S^1)^n$
(i.e., $|x_i| = |e^{\imag \theta}| = 1$).
Note that $\ake_G$ depends on $K = \{k_{ij} \}$ and $\colv{w}$,
and we will use the notation $\ake_{(G,K)}$ or $\ake_{(G,K,\colv{w})}$
when these dependencies are emphasized.

In much of this paper, we relax the root-counting problem
by considering all $\CC^*$-zeros of \eqref{equ: algebraic kuramoto}.
One important observation is that we can 
focus on graphs with no pendant (a.k.a. leaf) nodes.

\begin{lemma}\cite[Theorem 2.5.1]{juliathesis}\label{lem: leaf extension}
     Suppose $v \ne 0$ is a pendant node of $G$.
    Let $G' = G - \{ v \}$.
    Then any $\CC^*$-zero of $\aks_{G'}$
    extends to two distinct $\CC^*$-zeros for $\aks_G$.
\end{lemma}

\subsection{Kuramoto equations with phase delays}\label{sec: phase delays}
In models with phase delays, we may introduce parameters
$\{ \delta_{ij} \in \mathbb{R} \mid i \sim j \}$,
so that along an edge $\{i,j\}$,
oscillator $i$ responds not directly to the phase angle of oscillator $j$
but its delayed phase $\theta_j - \delta_{ij}$ \cite{yeung1999time}.
Then \eqref{equ: kuramoto sin} is generalized into:
\[
    0 = w_i - \overline{w} -
    \sum_{j \sim i} k_{ij}
    \sin(\theta_i - \theta_j + \delta_{ij}),
    \quad\text{for } i = 1,\ldots,n.
\]
Letting $C_{ij} = e^{\im \delta_{ij}}$, we can again make this system algebraic giving:
\begin{align}\label{equ: delayed Kuramoto}
    f_{G,i}(x_1,\ldots,x_n) =
    \overline{w}_i - \sum_{j \sim i} a_{ij} \left( \frac{x_i C_{ij}}{x_j} - \frac{x_j}{x_i C_{ij}} \right), \quad \text{for } i = 1,\ldots,n.
\end{align}
This system differs from \eqref{equ: algebraic kuramoto}
in the coefficients.
Yet, the same set of monomials are involved,
and as we will demonstrate,
the algebraic arguments we will develop
can be applied to this generalization.

\subsection{Power flow equations}\label{sec: power flow}

The \emph{PV power flow system} is one important variation
of the Kuramoto system.
In it, the graph $G$ models an electric power network
where $\nodes(G) = \{ 0, \ldots, n \}$ represent
buses in the power network.
An edge $\{i,j\} \in \edges(G)$, representing the connection between buses $i$ and $j$,
has a known complex admittance $g_{ij}' + \imag b_{ij}'$.
For each bus $i$, the relationship between its complex power injection $P_i + \imag Q_i$
and the complex voltages is
captured by the nonlinear equations
\begin{align}
    P_i &= \sum_{j \sim i} |V_i||V_j| (g_{ij}' \cos(\theta_i - \theta_j) + b_{ij}' \sin (\theta_i - \theta_j) ) \label{trigpfeqs1} \\
    Q_i &= \sum_{j \sim i} |V_i||V_j| (g_{ij}' \sin (\theta_i - \theta_j) - b_{ij}' \cos(\theta_i - \theta_j)) \label{trigpfeqs2}
\end{align}
where $|V_i|$ is the voltage magnitude at bus $i$
and $\theta_i$ is its phase angle.
We fix bus $0$ to be the \emph{slack bus} with $\theta_0 = 0$. For a complete treatment of the derivation of the power flow equations see \cite{grainger1994power}.

A node $i$ is a \emph{PV node} when
$Q_i$ and $\theta_i$ are unknown while $P_i$ and $|V_i|$ are known 
and it models a generator bus.
As above, with $x_i = e^{\imag \theta_i}$
we get the \term{PV algebraic power flow equations}
\begin{align}
   f_{G,i}(x_1,\ldots,x_n) &=
   P_i -  \sum_{j \sim i} g_{ij} \left(\frac{x_i}{x_j} + \frac{x_j}{x_i}\right)  + b_{ij}\left(\frac{x_i}{x_j} - \frac{x_j}{x_i}\right),
   \quad \text{for } i = 1,\ldots,n \label{eq:pfeqs}
\end{align}
where $b_{ij} = \frac{1}{2\imag}|V_i||V_j|b_{ij}'$ and $g_{ij} = \frac{1}{2}|V_i||V_j|g_{ij}'$.
When $g_{ij} = 0$, the corresponding power system is \emph{lossless}
and \eqref{eq:pfeqs} reduces to \eqref{equ: algebraic kuramoto}.
Otherwise, the system is \emph{lossy}
and \eqref{eq:pfeqs} differs from \eqref{equ: algebraic kuramoto}.
This is, again, a variation of the algebraic Kuramoto system \eqref{equ: algebraic kuramoto}
that involves the same monomials.

\subsection{BKK bound}

The root counting arguments in this paper revolve around the
Bernshtein-Kushnirenko-Khovanskii (BKK) bound,
especially Bernshtein's Second Theorem. 

\begin{theorem}[D. Bernshtein 1975~\cite{Bernshtein1975Number}]\label{thm:bernshtein-b}
    (A) For a square Laurent system $\colv{f} = (f_1,\dots,f_n)$,
    if for all nonzero vectors $\colv{v} \in \RR^n$,
    $\init_{\colv{v}}(\colv{f})$ has no $\CC^*$-zeros,
    then all $\CC^*$-zeros of $\colv{f}$ are isolated,
    and the total number, counting multiplicity,
    is the mixed volume $M = \operatorname{MV}(\newt(f_1),\ldots,\newt(f_n))$.
    
    (B) If $\init_{\colv{v}} (\colv{f})$ has a $\CC^*$-zero
    for some $\colv{v} \ne \colv{0}$,
    then the number of isolated $\CC^*$-zeros $\colv{f}$ has,
    counting multiplicity, is strictly less than $M$ if $M>0$.
\end{theorem}
A system for which condition (A) holds
is said to be \emph{Bernshtein-general}.
Only the special case of identical Newton polytopes,
i.e., when $\newt(f_1),\ldots,\newt(f_n)$ are all identical,
will be used.
This specialized version strengthens
Kushnirenko's Theorem \cite{Kushnirenko1975Newton}.

\subsection{Randomized algebraic Kuramoto system}\label{sec: randomized}

The analysis of the algebraic Kuramoto system~\eqref{equ: algebraic kuramoto}
can be further simplified through ``randomization''.
For any nonsingular square matrix $R$,
the systems $\aks_G$ and $\raks_G := R \, \aks_G$ have the same zero set.
With a generic choice of $R$,
there will be no complete cancellation of terms, 
and $\raks_G$ will be referred to as the \term{randomized (algebraic) Kuramoto system}.
This system is \emph{unmixed} in the sense that
$\rake_{G,1},\ldots,\rake_{G,n}$ have identical supports
since they involve the same set of monomials, namely,
constant terms and
$\{ x_i x_j^{-1}$ , $x_j x_i^{-1} \}_{\{i,j\} \in \edges(G)}$.

The randomization $\aks_G \mapsto \raks_G$ does not alter the zero set
but makes a very helpful change to the tropical structure:
There is a mapping between the
graph-theoretical features of $G$
and the tropical structures of $\raks_G$
through which we can gain key insight into the structure
of the zeros of $\raks_G$.

\subsection{Adjacency polytopes}

For each $i=1,\ldots,n$, $\supp(\rake_{G,i})$ are all identical
and given by
\[
    \radjp_G :=
    \{ 
        \pm (\colv{e}_i - \colv{e}_j) \mid 
        i \sim j
    \}
    \;\cup\;
    \{ \colv{0} \},
\]
where $\colv{e}_i$ is the $i$-th standard basis vector of $\RR^n$ for $i=1,\ldots,n$,
and $\colv{e}_0 = \colv{0}$.
The $\conv(\radjp_G)$ is the
\emph{adjacency polytope} or \emph{symmetric edge polytope} of $G$
\cite{Chen2019Unmixing,ChenDavisMehta2018Counting,ChenMehta2017Network,MatsuiHigashitaniNagazawaOhsugiHibi2011Roots},
and is related to \emph{root polytopes} \cite{Postnikov2009Permutohedra}.
It has appeared in number theory and discrete geometry
(see the overview provided in \cite{DAliDelucchiMichalek2022Many}).

We will not distinguish $\radjp_G$ from $\conv(\radjp_G)$,
e.g., a ``face'' of $\radjp_G$ refers to a subset $F \subseteq \radjp_G$
such that $\conv(F)$ is a face of $\conv(\radjp_G)$.
The facets and boundary of $\radjp_G$ are denoted
$\facets(\radjp_G)$ and $\partial \radjp_G$, respectively.
$\radjp_G$ is centrally symmetric,
and both $\radjp_G$ and $\partial \radjp_G$ have unimodular triangulations.

By Kushnirenko's Theorem \cite{Kushnirenko1975Newton},
$\nvol(\radjp_G)$ is an upper bound for the $\CC^*$-root count
for $\raks_G$ and $\aks_G$
and the real root count to \eqref{equ: kuramoto sin}.
This is the \emph{adjacency polytope bound}
\cite{Chen2019Unmixing,ChenDavisMehta2018Counting}.

\subsection{Faces and face subgraphs}

There is an intimate connection between faces of $\radjp_G$
and subgraphs of $G$.
Since $\colv{0}$ is an interior point of $\radjp_G$,
every vertex of a proper face $F$
of $\radjp_G$
is of the form $\colv{e}_i - \colv{e}_j$ for some $\{i,j\} \in \edges(G)$.
Thus, it is natural to consider the corresponding
\emph{facial subgraph} $G_F$ and 
\emph{facial subdigraph} $\dig{G}_F$ involving a subsets of nodes
and edges sets
\begin{align*}
    \edges(G_F)       &= \{ \{i,j\} \mid \colv{e}_i - \colv{e}_j \in F \text{ or } \colv{e}_j - \colv{e}_i \in F \} 
    \;\text{and} \\
    \edges(\dig{G}_F) &= \{  (i,j)  \mid \colv{e}_i - \colv{e}_j \in F \}.
\end{align*}
As defined in \cite{Chen2019Directed,DAliDelucchiMichalek2022Many}, for $F \in \facets(\radjp_G)$,
$G_F$ is called a \emph{facet subgraph}
and $\dig{G}_F$ is a \emph{facet subdigraph}.
Higashitani, Jochemko, and Micha\l{}ek
provided a topological classification of face subgraphs
\cite[Theorem 3.1]{HigashitaniJochemkoMichalek2019Arithmetic}
and it was later reinterpreted \cite[Theorem 3]{ChenDavisKorchevskaia2022Facets}.
We state the latter here.

\begin{theorem}[Theorem 3 \cite{ChenDavisKorchevskaia2022Facets}]\label{thm: faces are max bipartite}
    Let $H$ be a nontrivial connected subgraph of $G$.
    \begin{enumerate}
        \item $H$ is a face  subgraph of $G$ if and only if it is a maximal bipartite subgraph of $G[\nodes(H)]$.
        \item $H$ is a facet subgraph of $G$ if and only if it is a maximal bipartite subgraph of $G$.
    \end{enumerate}
\end{theorem}

Here, $G[V]$ is the subgraph induced by the subset $V \subset \nodes(G)$.
Multiple faces can correspond to the same facial subgraph.
The crisper parameterization is given by the correspondence
$F \mapsto \dig{G}_F$.
Reference \cite{ChenDavisKorchevskaia2022Facets} describes
a necessary balancing conditions for facial subdigraphs.

For a facial subdigraph $\dig{G}_F$, its \emph{reduced incidence matrix}
$\rinc(\dig{G}_F)$ is the matrix with columns $\colv{e}_i - \colv{e}_j$
for $(i,j) \in \edges(\dig{G}_F)$. 
Its null space can be interpreted as the space of circulations of $\dig{G}_F$.

For a subdigraph $\dig{H}$ of $\dig{G}$,
the \emph{coupling vector} $\rowv{k}(\dig{H})$ has entries 
$k_{ij}$ for $(i,j) \in \dig{H}$.
Similarly, the entries of $\rowv{a}(\dig{H})$ are
the complexified coupling coefficients $a_{ij} = \frac{k_{ij}}{2 \imag}$.
The ordering of the entries is arbitrary,
but when appearing in the same context with $\rinc(\dig{H})$,
consistent ordering is implied.

\subsection{Facial systems}

The vast literature on the facial structure of $\radjp_G$
gives us a shortcut to understanding the initial systems of $\raks_G$,
since
they have particularly simple descriptions
corresponding to proper faces of $\radjp_G$.
For any $0 \neq \rowv{v} \in \RR^n$,
the initial system $\init_{\rowv{v}}(\raks_G)$ is
\begin{equation}\label{equ: facial system}
    \init_{\rowv{v}}(\rake_{G,i})(\rowv{x}) =
    \sum_{ \colv{e}_j - \colv{e}_{j'} \in F } c_{i,j,j'} \, \rowv{x}^{\colv{e}_j - \colv{e}_{j'}} =
    \sum_{ (j,j') \in \edges(\dig{G}_F) } c_{i,j,j'} \, \rowv{x}^{\colv{e}_j - \colv{e}_{j'}}
    \quad\text{for } i = 1,\ldots,n,
\end{equation}
where $F$ is the face of $\radjp_G$ for which
$\rowv{v}$ is an inner normal vector.
We will make frequent use of this geometric interpretation and
therefore it is convenient to slightly abuse the notation and write
$\init_F(\raks_G) := \init_{\rowv{v}} (\raks_G)$.
It will be called a \emph{facial system} of $\raks_G$,
(or a \emph{facet system} if $F \in \facets(\radjp_G)$).

\section{Generic $\CC^*$-root count}\label{sec: generic root count}

In this section, we establish the $\CC^*$-root count for the algebraic Kuramoto system \eqref{equ: algebraic kuramoto}
for generic choices of real or complex parameters
$\{\overline{w}_i\}_{i \in \nodes(G)}$ 
and $\{ k_{ij} \}_{\{i,j\} \in \edges(G)}$.
Since $\RR$ is Zariski-dense in $\CC$, 
it is sufficient to consider generic choices of complex parameters.
We show that the generic $\CC^*$-root count is the adjacency polytope bound
$\nvol(\radjp_G)$. 
A corollary is that this system is
Bernshtein-general, despite the constraints on the coefficients.

There are three main obstacles.
First, in each polynomial the monomials $x_i/x_j$ and $x_j/x_i$ share the same coefficient $a_{ij}$.
Second, for any edge $\{i,j\} \in \edges(G)$,
the $i$-th and $j$-th polynomials have the terms
$a_{ij} (x_i x_j^{-1} - x_j x_i^{-1})$
and
$a_{ji} (x_j x_i^{-1} - x_i x_j^{-1})$,
respectively, which are negations of each other
since $a_{ij} = a_{ji}$.
Therefore, the allowed choices of coefficients consists of a
nowhere dense subset of (Lebesgue) measure 0
in the space of all possible complex coefficients.
    Finally, unless $G$ is the complete graph,
    the Newton polytopes of the algebraic Kuramoto system
    are not full-dimensional
    which prevents simpler arguments (e.g. \cite{Chen2018Equality}) from being applied.
We will show that despite these,
the maximum $\CC^*$-root count, given by the adjacency polytope bound is generically attained.

Before presenting the main theorem of this section,
we first establish a few technical results that will be used,
some of which are well known but are nonetheless included for completeness.

\begin{lemma}\label{lem: facet subdivision}
    The lifting function $\tilde{\omega} : \radjp_G \to \QQ$, given by
    \[
        \tilde{\omega}(\colv{a}) =
        \begin{cases}
            0 & \text{if } \colv{a} = \colv{0} \\
            1 & \text{otherwise}
        \end{cases}
    \]
    induces a regular subdivision
    $\subdiv_{\tilde{\omega}}(\radjp_G) =  \{ \colv{0} \cup F \mid F \in \facets(\radjp_G) \}$.
\end{lemma}

This is a well known consequence of assigning a sufficiently small lifting value
to an interior point of a point configuration ($\colv{0}$ in this case).
The resulting regular subdivision may be too coarse to be useful in our discussions.
Indeed, it will not be a triangulation unless $G$ is has no even cycles
\cite{ChenDavis2022Toric,ChenDavisKorchevskaia2022Facets}.
It can be refined into a triangulation through perturbations on
the nonzero lifting values.

\begin{lemma}\label{lem: generic lifting}
    For generic but symmetric choices
    $\{ \delta_{ij} = \delta_{ji} \in \RR \mid \{i,j\} \in \edges(G) \}$
    that are sufficiently close to 0, the function
    $\omega : \radjp_G \to \QQ$, given by
    \[
        \omega(\colv{a}) =
        \begin{cases}
            0               & \text{if } \colv{a} = \colv{0} \\
            1 + \delta_{ij} & \text{if } \colv{a} = \colv{e}_i - \colv{e}_j
        \end{cases}
    \]
    induces a regular unimodular triangulation
    $\Delta_\omega = \subdiv_\omega(\radjp_G)$
    that is a refinement of $\subdiv_{\tilde{w}}(\radjp_G)$
    and its cells are in one-to-one correspondence with cells in
    a unimodular triangulation of \(\partial \radjp_G\).
    Indeed, each cell is of the form
    $0 \cup \Delta$ where $\Delta$ is a simplex in $\partial \radjp_G$.
\end{lemma}

\begin{proof}
    We first show the interior point $\colv{0}$ is contained in every cell. 
    Fix a $C \in \Delta_\omega$ and
    let $(\rowv{v},1)$ be an inner normal vector
    of the lower facet of $\radjp_G^\omega$
    whose projection is $C$.
    Suppose $\colv{0} \not\in C$, 
    then $C$ contains an affinely independent set of $n+1$ points
    $\{ \colv{a}_0,\ldots,\colv{a}_n \} \not\ni \colv{0}$,
    and $\colv{v}$ satisfies the equation
    \[
        \begin{bmatrix}
            \colv{a}_1^\top - \colv{a}_0^\top \\
            \vdots \\
            \colv{a}_n^\top - \colv{a}_0^\top \\
        \end{bmatrix}
        \,
        \rowv{v}^\top
        =
        \begin{bmatrix}
            \omega(\colv{a}_0) - \omega(\colv{a}_1) \\
            \vdots \\
            \omega(\colv{a}_0) - \omega(\colv{a}_n)
        \end{bmatrix}.
    \]
    Let $B$ be the matrix on the left and $\colv{\beta}$ be the vector on the right,
    then $\rowv{v}^\top = B^{-1} \colv{\beta}$ and thus
    \[
        | \inner{ \rowv{v} }{ \colv{a} } | \le 
        \| \rowv{v}^\top \| \| \colv{a} \| \le
        \| \colv{a} \| \, \| B^{-1} \| \, \| \colv{\beta} \|
        \quad\text{for any } \colv{a} \in C.
    \]
    Note that entries of $\colv{\beta}$ are differences among the
    $\{ \delta_{ij} \}$.
    Therefore, for any $\epsilon > 0$, there is a $\delta$ such that
    $|\delta_{ij}| < \delta$ for all $\{i,j\} \in \edges(G)$
    implies $| \inner{ \rowv{v} }{ \colv{a} } | < \epsilon$,
    contradicting with the assumption that
    \[
        \inner{ \rowv{v} }{ \colv{a} } + \omega(\colv{a}) <
        \sinner{ \rowv{v} }{ \colv{0} } + \omega(\colv{0}) = 0.
    \]
    This shows that $\colv{0}$ must be contained in every cell.

    To show $\Delta_\omega$ is a triangulation,
    it is sufficient to show nonzero points in a cell $C \in \Delta_\omega$
    are assigned independent lifting values by $\omega$.
    If both $\pm (\colv{e}_i - \colv{e}_j) \in C$,
    for some $\{i,j\} \in \edges(G)$,
    then
    \[
        \inner{ \rowv{v} }{ \pm (\colv{e}_i - \colv{e}_j) } + 
        \omega( \pm (\colv{e}_i - \colv{e}_j) ) = h,
    \]
    where
    $h = \min \{
        \inner{ \rowv{v} }{ \colv{a} } + \omega(\colv{a})
        \mid \colv{a} \in \radjp_G
    \}$.
    Since $\omega(\pm (\colv{e}_i - \colv{e}_j)) = 1 + \delta_{ij}$,
    summing the two produces
    \[
        h = 1 + \delta_{ij} > 0,
    \]
    which contradicts the constraint that
    $0 = \sinner{\colv{v}}{\colv{0}} + \omega(\colv{0}) \ge h$
    for $\delta_{ij}$'s sufficiently close to 0.
    Therefore, if $\colv{e}_i - \colv{e}_j \in C$,
    then $\colv{e}_j - \colv{e}_i \not\in C$.
    Consequently, the nonzero points in $C$ are associated with
    independent and generic choices of lifting values,
    and thus $C$ must be a simplex.

    For unimodularity, note that, since
    $\colv{e}_i - \colv{e}_j \in C$ implies $\colv{e}_j - \colv{e}_i \not\in C$
    for $\{i,j\} \in \edges(G)$ and for a cell $C \in \Delta_\omega$,
    the nonzero points of $C$, as vectors, are exactly columns of (signed) incidence matrix of $G$,
    which is unimodular.
    Therefore, each cell of $\Delta_\omega$ is unimodular.
    
    Finally, to show that $\Delta_\omega$ is a refinement of
    $\subdiv_{\tilde{\omega}}(\radjp)$ from \Cref{lem: facet subdivision},
    we fix an arbitrary ordering of the points in $\radjp_G$
    and consider lifting functions $\omega$ and $\tilde{\omega}$,
    as vectors in $\RR^{|\radjp_G|}$.
    Let $\mathcal{C} = \seccone(\radjp_G,\Delta_\omega)$ be the (closed) secondary cone
    of $\Delta_\omega$ in $\radjp_G$.
    Then $\mathcal{C}$ is full-dimensional,
    since $\Delta_\omega$ is a triangulation.
    By assumption, $\omega$ is sufficiently close to $\tilde{\omega}$,
    so $\tilde{\omega} \in \mathcal{C}$,
    and $\Delta_\omega = \subdiv_\omega(\radjp_G)$
    equals or refines $\subdiv_{\tilde{\omega}}(\radjp_G)$.
\end{proof}

A lifting function \(\omega : \radjp_G \to \mathbb{Q}\)
for which \Cref{lem: generic lifting} holds will be called a
\term{generic symmetric lifting function} for $\radjp_G$.
We now turn our attention to its graph-theoretic consequences.

\begin{restatable}{lemma}{simplextree}\label{lem: simplex tree}
    For a generic symmetric lifting function $\omega$ for $\radjp_G$,
    let $\Delta$ be a simplex in $\Delta_\omega$.
    Then the digraph $\dig{G}_\Delta$
    is acyclic, and its underlying graph $G_\Delta$
    is a spanning tree of $G$.
\end{restatable}

The proof of this lemma is nearly identical to the proof of
\cite[Theorem 1]{Chen2019Directed}
and we include an elementary proof in \Cref{sec: lemmas} for completeness.
Finally, we show a facet system derived from a spanning tree
has exactly one $\CC^*$-root, and it is nonsingular.

\begin{lemma}\label{lem: tree system}
    Suppose $\dig{T} < \dig{G}$ is a acyclic,
    and its underlying graph $T$ is a spanning tree of $G$,
    then, for generic $\overline{w}_1,\ldots,\overline{w}_n \in \CC^*$
    and \emph{any} choices of $a_{ij} \in \CC^*$,
    the system of $n$ Laurent polynomials
    \[
        \overline{w}_i
        - \sum_{j  :   i \rightsquigarrow j} a_{ij} \frac{x_i}{x_j}
        + \sum_{j  :   j \rightsquigarrow i} a_{ij} \frac{x_j}{x_i}
        \quad\text{for } i = 1,\ldots,n
    \]
    has a unique zero in $(\CC^*)^n$,
    and this zero is isolated and regular.
    Here, $j \rightsquigarrow i$ means $(j,i) \in \edges(\dig{T})$.
\end{lemma}

\begin{proof}
    We prove this by induction on the number of vertices of $G$, $N$. Denote the above system as $\colv{f}_{\dig T}(x_0,\ldots,x_{N-1})$.
    The statement is  true for the case
    where $N=2$ and $| \nodes(G) | = | \nodes(T) | =2$.
    
     Assume the statement is true for any graph with $N$ nodes
    and consider the case $G$ has $N+1$ nodes. 
    By re-labeling, we can assume node $N$ is a leaf in $T$
    and is adjacent to node $N-1$.
    Since this system is homogeneous of degree 0,
    we can deviate from our convention that $x_0 = 1$
    and scale the homogeneous coordinates so that $x_{N-1} = 1$. 
    Then the above system is decomposed as the Laurent system
    $\colv{f}_{\dig T'}(x_0,\ldots,x_{N-1})$
    and the binomial
    $\overline{w}_{N} \pm a_{N,N-1} x_N^{\mp 1}$,
    where $\dig{T'} = \dig{T} - \{N\}$. 
   The induction hypothesis, that
   $\dig{f}_{\dig{T'}}(x_0,\ldots,x_N)$ has a unique isolated and regular
   thus completes the proof.
\end{proof}

With these technical preparations,
we now establish the main theorem of this section.

\begin{theorem}\label{thm: generic root count}
    For generic choices of real or complex constants
    $\overline{w}_1,\ldots,\overline{w}_n$
    and generic but symmetric choices of real or complex coupling coefficients
    $\{ k_{ij} = k_{ji} \ne 0 \mid \{i,j\} \in \edges(G) \}$,
    the $\CC^*$-zeros of the algebraic Kuramoto system
    \eqref{equ: algebraic kuramoto}
    are isolated and nonsingular,
    and the total number is $\nvol(\radjp_G)$.
\end{theorem}

This is a far generalization of 
\cite[Corollary 9 and Theorem 16]{ChenDavisMehta2018Counting} and 
\cite[Theorem 3.3.3]{GuoSalam1990Determining}
where the generic $\CC^*$-root count is established for tree\footnote{
   The generic number of complex equilibria for the Kuramoto model on tree networks is $2^n$.
    This fact appears to be well known among researchers in power systems
    long before the referenced paper \cite{ChenDavisMehta2018Counting}.
},
cycle, and complete networks only.

\begin{proof}
    By Sard's Theorem, for generic choices
    of $\overline{w}_1,\ldots,\overline{w}_n$, the $\CC^*$-zero set of $\aks_G$
    consists of isolated and regular points.
    Under this assumption, we only need to establish the $\CC^*$-root count.
    Moreover, since it is already known this root count is bounded by $\nvol(\radjp_G)$,
    it is sufficient to show that it is greater than or equal to this bound.
    We shall take a constructive approach
    through a specialized version of the polyhedral homotopy of Huber and Sturmfels
    \cite{HuberSturmfels1995Polyhedral}.

    Let $\omega : \radjp_G \to \QQ$ be a generic symmetric lifting function
    for $\radjp_G$. 
    We use the notation $\omega_{ij} = \omega(\colv{e}_i - \colv{e}_j) = \omega(\colv{e}_j - \colv{e}_i)$
    and define the function $\colv{h} = (h_1,\ldots,h_n) : (\CC^*)^n \times \CC \to \CC^n$,
    given by
    \[
        h_i (x_1,\ldots,x_n,t) =
        \overline{w}_i - \sum_{j \sim i}
        a_{ij} t^{\omega_{ij}}
        \left(
            \frac{x_i}{x_j} -
            \frac{x_j}{x_i}
        \right).
    \]
    Away from $t=0$, $\colv{h}$ is a parameterized version of the original
    system $\aks_G$ with coefficients being
    analytic functions of the parameter $t$.
    Note that since $\omega_{ij} = \omega_{ji}$, for any choice of
    $t \in \CC$, the system still satisfies the symmetry constraints on the coefficients.
    By the Parameter Homotopy Theorem \cite{MorganSommese1989Coefficient},
    for $t$ outside a proper Zariski closed (i.e., finite) set $Q \subset \CC$,
    the number of isolated $\CC^*$-zeros of  $\colv{h}(\,\cdot\,,t) =0$ is a constant
    which is also an upper bound for the isolated $\CC^*$-root count
    for $\colv{h}(\,\cdot\,,t)$ for all $t \in \CC$.
    Let $\eta$ be this generic $\CC^*$-root count.
    Since $\overline{w}_i$ and $a_{ij}$ are chosen generically,
    and $\colv{h}(\,\cdot\,,1) \equiv \aks_G(\;\cdot\;)$,
    we can conclude that  $t = 1$ is outside $Q$ and thus
    $\eta$ agrees with the generic $\CC^*$-root count of $\aks_G$
    that we aims to establish.
    That is, it is sufficient to show $\eta \ge \nvol(\radjp_G)$. 

    Taking a constructive approach,
    we now construct $\eta$ smooth curves (of one real-dimension),
    that will connect the $\CC^*$-zeros of $\aks_G$
    and the collection of $\CC^*$-zeros of the special systems described in \Cref{lem: tree system}.
    This connection allows us to count the number of curves.
    
    Along a ray $t : (0,1] \to \CC$ on the complex plane parameterized by
    $
        t(\tau) = \tau e^{i \theta}
    $
    for a choice of $\theta \in [0,2\pi)$ that avoids $Q$
    (i.e., $t(\tau) \not\in Q$ for all $\tau \in (0,1]$),
    the function $\colv{h}$ is given by
    \[
        h_i (x_1,\ldots,x_n,t) =
        h_i (x_1,\ldots,x_n, \tau e^{i \theta}) =
        \overline{w}_i - \sum_{j \sim i}
        a_{ij} e^{i \theta \omega_{ij}}
        \tau^{\omega_{ij}}
        \left(
            \frac{x_i}{x_j} -
            \frac{x_j}{x_i}
        \right),
    \]
    and $\colv{h}$  has $\eta$ isolated and regular $\CC^*$-zeros
    for all $\tau \in (0,1]$.
    By the principle of homotopy continuation,
    the zero set of $\colv{h}$ form $\eta$ smooth curves in $(\CC^*)^n \times (0,1]$
    smoothly parameterizable by $\tau$. 
    The problem is now reduced to counting these curves.

    Fixing such a curve $C$, the asymptotic behavior of $C$ as $\tau \to 0$,
    in a compactification of $(\CC^*)^n$ can be characterized by
    the solutions to an initial system of $\colv{h}$,
    as a system of Laurent polynomials in 
    $\CC\{\tau\}[x_1^{\pm 1}, \ldots, x_n^{\pm 1}]$.
    Stated in affine coordinates, by the Fundamental Theorem of Tropical Geometry,
    there exists a system of convergent Puiseux series
    $x_1(\tau),\ldots,x_n(\tau) \in \CC\{\tau\}$
    that represents the germ of $C$, as an analytic variety, at $\tau=0$,
    such that the leading coefficients,
    as a point in $(\mathbb{C}^*)^n$ satisfies
    the initial system $\init_{\rowv{v}}(\colv{h})$,\footnote{%
        We do not claim $\colv{h}$ form a tropical basis,
        nor do we require that $\init_{\rowv{v}}(\colv{h})$ generate the corresponding initial ideal.
        Yet, as we will show, an initial system of $\colv{h}$
        of the special form we will describe
        is sufficient to uniquely determine the leading coefficients
        of the Puiseux series expansion of a solution path.
    }
    and $\rowv{v} = (v_1,\ldots,v_n) \in \QQ^n$
    are the orders of the Puiseux series
    $x_1(\tau),\ldots,x_n(\tau) \in \CC\{\tau\}$.
    Therefore, it is sufficient to show that
    there are at least $\nvol(\radjp_G)$ distinct initial systems of $\colv{h}$
    that will each contribute one solution path.

    By \Cref{lem: generic lifting},
    the regular subdivision $\Delta_\omega = \subdiv_\omega(\radjp_G)$
    is a unimodular triangulation.
    Fix a simplex $\Delta \in \Delta_\omega$,
    and let $(\rowv{v},1)$ be the upward pointing inner normal vector
    that defines $\Delta$,
    then by construction, there is an affinely independent set
    $\{\colv{a}_1,\ldots,\colv{a}_n\} \subset \partial \radjp_G$
    such that
    \begin{align*}
        \inner{ \rowv{v} }{ \colv{a}_k } + \omega(\colv{a}_k) &= 0
        \quad\text{for } k = 1,\ldots,n,\; and
        \\
        \inner{ \rowv{v}   }{ \colv{a} } + \omega(\colv{a}) &> 0
        \quad\text{for all } \colv{a} \in \radjp_G \setminus \{ \colv{0}, \colv{a}_1,\ldots,\colv{a}_n \}.
    \end{align*}
    Therefore, the exponent vectors associated with monomials with
    nonzero coefficients in $\init_{\rowv{v}} (\colv{h})$ are exactly
    $\{ 0, \colv{a}_1,\ldots,\colv{a}_n \}$.
    By \Cref{lem: simplex tree}, $\dig{T} = \dig{G}_\Delta$ is
    an acyclic subdigraph and its associated subgraph $T$
    is a spanning tree of $G$.
    Indeed, $\init_{\rowv{v}} (\colv{h})$ consists of Laurent polynomials
    \[
        \overline{w}_i
        - \sum_{j \in \mathcal{N}^+_{\vec{T}}(i)} a_{ij} \frac{x_i}{x_j}
        + \sum_{j \in \mathcal{N}^-_{\vec{T}}(i)} a_{ij} \frac{x_j}{x_i}
        \quad\text{for } i = 1,\ldots,n.
    \]
    By \Cref{lem: tree system}, this system has a unique $\CC^*$-zero,
    which is regular.
    Following from the principle of homotopy continuation,
    it gives rise to a smooth curve defined by
    $\colv{h}(\rowv{x},\tau e^{i\theta}) = 0$ in $(\CC^*)^n \times (0,1]$.

    This holds for every simplex in $\Delta_\omega$,
    i.e., each simplex $\Delta \in \Delta_\omega$
    contributes a curve define by $\colv{h} = \colv{0}$.
    Moreover, these curves are distinct,
    so 
    $
        \eta \ge |\Delta_\omega| = \nvol(\radjp_G),
    $
    which completes the proof.
\end{proof}

This establishes the fact that
the generic $\CC^*$-root count for the algebraic Kuramoto system $\aks_G$
is exactly the adjacency polytope bound,
despite the algebraic constraints on the parameters.
Note that the BKK bound for $\aks_G$ is always between
the generic $\CC^*$-root count and the adjacency polytope bound.
Therefore, we can indirectly derive the Bernshtein-genericity
of $\aks_G$.

\begin{corollary}\label{cor:bernshtein_general}
    For generic real or complex $\overline{w}_1,\ldots,\overline{w}_n$
    and generic but symmetric real or complex $k_{ij} = k_{ji} \ne 0$
    for $\{i,j\} \in \edges(G)$,
    the algebraic Kuramoto system \eqref{equ: algebraic kuramoto} is Bernshtein-general,
    and
    \[
        \mvol(\newt(\ake_{G,1}),\ldots,\newt(\ake_{G,n})) =
        \nvol(\conv(\newt(\ake_{G,1}) \cup \cdots \cup \newt(\ake_{G,n}))) =
        \nvol(\radjp_G).
    \]
\end{corollary}

\begin{remark}
    We remark that the proof of \Cref{thm: generic root count}
    utilizes many ideas from the polyhedral homotopy of
    Huber and Sturmfels \cite{HuberSturmfels1995Polyhedral}.
    The main differences are that we consider a lifting function
    that preserves the relationships among the coupling coefficients
    and we choose generators of the ideals at ``toric infinity'' with a nice tree structure
    instead of requiring them to be binomial.
    
    Finally, we draw attention to the tropical nature of this proof.
    The first half of the proof establishes the tropical version
    of the generic root count result:
    With the assignment of the valuation
    $\operatorname{val}(\overline{w}_i) = 0$ and generic but symmetric
    $\operatorname{val}(a_{ij}) = \operatorname{val}(a_{ji}) = \omega_{ij} = \omega_{ji}$,
    we showed that the intersection number of the tropical hypersurfaces
    defined by $f_{G,1},\ldots,f_{G,n}$ is bounded below by $\nvol(\radjp)$
    which is also the stable self-intersection number
    of the tropical hypersurface defined by the randomized Kuramoto system
    (as defined in \Cref{sec: randomized}) with the same valuation.
    Computing generic root counts via stable tropical intersection
    is the topic of a recent paper by
    Paul Helminck and Yue Ren \cite{HelminckRen2022Generic},
    in which the special case for complete networks is studied as an example.
\end{remark}

Another consequence of \Cref{thm: generic root count}
is the monotonicity of the generic $\CC^*$-root count,
since the volume of the adjacency polytope is strictly increasing
as new edges and nodes are added.

\begin{corollary}
    Let $(G_1,K_1,\colv{w}_1)$ and $(G_2,K_2,\colv{w}_2)$ be two
    connected networks such that $G_1 < G_2$.
    Then the generic $\CC^*$-root count of $\aks_{G_1}$
    is strictly less than that of $\aks_{G_2}$.
\end{corollary}

In addition, this shows that
the phase-delayed algebraic Kuramoto system is Bernshtein-general. 

\begin{corollary}\label{cor phase delay}
    For generic choices of $\overline{w}_i, a_{ij}, C_{ij}$,
    the algebraic Kuramoto systems with phase delays
    \eqref{equ: delayed Kuramoto}
    is Bernshtein-general,
    and its $\CC^*$-root count equals to the normalized volume of $\nabla_G$.
\end{corollary}
\begin{proof}
    As noted before, for generic choices of $\overline{w}_1,\ldots,\overline{w}_n$,
    all $\CC^*$-zeros of this system are isolated and regular.
    Let $\eta$ be the generic number of $\CC^*$-zeros this system has.
    By \Cref{thm:bernshtein-b} and \Cref{cor:bernshtein_general},
    we know $\eta \leq \nvol(\radjp_G)$.
    By the Parameter Homotopy Theorem \cite{MorganSommese1989Coefficient},
    the number of $\CC^*$-zeros of this parameterized system constant on
    a nonempty Zariski open set of the parameters.
    By setting $C_{ij} = 1$, we recover \eqref{equ: algebraic kuramoto}
    and by \Cref{thm: generic root count} and \Cref{cor:bernshtein_general},
    this system is Bernshtein-general,
    with $\nvol(\radjp_G)$ $\CC^*$-zeros.
    This shows that $\eta \geq \nvol(\radjp_G)$, giving the result.
\end{proof}

From the same argument, we get the generic root count 
for the PV power flow equations.

\begin{corollary}
    For generic choices of $P_i, b_{ij},g_{ij}$,
    the lossy PV-type algebraic power flow system \eqref{eq:pfeqs}
    is Bernshtein-general,
    and its $\CC^*$-root count equals to the normalized volume of $\radjp_G$.
\end{corollary}

\begin{table}
    \centering
    \begin{tabular}{p{8cm}p{0.8cm}p{6.3cm}}\toprule
        Network & & Volume of adjacency polytope \\ \midrule
        A tree network with $N$ nodes \cite{ChenDavisMehta2018Counting} & & $2^{N-1}$ 
        \\ \midrule
        A cycle network with $N$ nodes \cite{ChenDavisMehta2018Counting} & & $ N \binom{N-1}{ \lfloor (N-1) / 2 \rfloor } $ 
        \\ \midrule
        A complete network with $N$ nodes \cite{BaillieulByrnes1982Geometric} & & $  \binom{2N-2}{ N-1 } $ 
        \\ \midrule
        A network formed by joining $k$ cycles of lengths
        $2m_1, \ldots, 2m_k$, consecutively along an edge \cite{DAliDelucchiMichalek2022Many} & &
        $\frac{1}{2^{k-1}} \prod_{i=1}^k m_i \binom{2m_i}{m_i}$ 
        \\ \midrule 
        A network formed by joining
        two odd cycles of lengths $2m_1+1$ and $2m_2+1$ along an edge \cite{DAliDelucchiMichalek2022Many}  & &
        $(m_1 + m_2 + 2m_1m_2) \binom{2m_1}{m_1} \binom{2m_2}{m_2}$  
        \\ \midrule
        A wheel graph with $N+1$ nodes \cite{DAliDelucchiMichalek2022Many}& &
        	$\begin{cases}
        		(1-\sqrt{3})^{N} + (1+\sqrt{3})^{N}     & \text{$N$ odd} \\
        		(1-\sqrt{3})^{N} + (1+\sqrt{3})^{N} - 2 & \text{$N$ even}
        	\end{cases}$
        \\
        \bottomrule
    \end{tabular}
    \caption{Known results for the volume of adjacency polytopes for families of networks.}
    \label{table:adj polytope vol}
\end{table}

We conclude this section by showing some known results
for volumes of adjacency polytopes in \Cref{table:adj polytope vol}.
We observe that for cyclic, tree and wheel networks,
previous results by convex geometers give sharp bounds
for the number of complex roots to the Kuramoto equations.
Conversely, power system engineers have known for decades the
complex root count of the PV power flow equations for complete networks \cite{BaillieulByrnes1982Geometric}. 
Now due to \Cref{thm: generic root count},
this gives the volume of adjacency polytopes of complete networks.
This demonstrates how \Cref{thm: generic root count} 
bridges the two worlds.

\section{Explicit genericity conditions on coupling coefficients}
\label{sec: explicit genericity conditions}

So far, we have established that the generic $\CC^*$-root count of
the algebraic Kuramoto system $\aks_G$
is the adjacency polytope bound, $\nvol(\radjp_G)$.
That is, for almost all choices of the parameters
$\vec{w}$ (natural frequencies) and $K$ (coupling coefficients),
the $\CC^*$-root count for $\aks_G$ is equal to $\nvol(\radjp_G)$.
This section aims to understand exactly when
the $\CC^*$-root count drops below $\nvol(\radjp_G)$. 
We focus on the effect of coupling coefficients $\{ k_{ij} \}$.
The effect of natural frequencies $\vec{w}$
will be explored in \Cref{sec: +dimensional}.

\begin{definition}[Exceptional coupling coefficients]\label{def: global exceptional}
    Given a connected graph $G$ with $n+1$ nodes,
    we define its \term{exceptional coupling coefficients},
    $\badcoupling(G)$, to be the set of
    symmetric and nonzero coupling coefficients $K = \{ k_{ij} \}_{i \sim j}$
    such that, counting multiplicity, the number of isolated $\CC^*$-zeros of  $\aks_{(G,K,\colv{w})}$
    is strictly less than $\nvol(\radjp_G)$, for any choice of $\colv{w} \in \CC^n$. 
\end{definition}

In other words, $\badcoupling(G)$ is the set of symmetric coupling coefficients
for which the algebraic Kuramoto system \eqref{equ: algebraic kuramoto}
is not Bernshtein-general.
In the following subsections, we will develop an
algebraic and graph-theoretic description for
the exceptional coupling coefficients $\badcoupling(G)$.

We start with the analysis of a single facial system.
For a face $F$ of $\radjp_G$,
we describe the coupling coefficients
for which the facial system $\init_F (\raks_G)$ has a $\CC^*$-zero,
signaling the genericity condition for \Cref{thm: generic root count} is broken
and the number of isolated $\CC^*$-zeros of $\aks_{(G,K)}$ drops.

\begin{definition}[Exceptional coupling coefficients for a face]\label{def: bad coupling}
    For a proper face of $\radjp_G$, we define
    \[
        \badcoupling(\dig{G}_F) = \{
            \rowv{k}(\dig{G}_F) \in (\CC^*)^{|F|} \mid
            \init_F(\raks_G) = \colv{0} \ \text{has a } \CC^*\text{-solution}
        \}
    \]
    to be the set of \term{exceptional coupling coefficients} with respect to the face $F$
    (or $\dig{G}_F$).
\end{definition}

It is worth noting that such exceptional coupling coefficients are extremely rare.
By \Cref{thm: generic root count,thm:bernshtein-b},
$\badcoupling(\dig{G}_F)$ is contained in a
proper and Zariski closed subset of Lebesgue measure zero.
Yet, a full description of $\badcoupling(\dig{G}_F)$ is possible,
as we will detail in the rest of this section.
We start, in \Cref{sec: circulation form},
with a ``circulation'' formulation of a facial system,
which bridges together the algebraic and graph-theoretic
sides of this problem.
In \Cref{sec: tree solution},
we describe a parameterization of
the $\CC^*$-zero set of a facial system.
We use this parameterization in \Cref{sec: facial balancing conditions} to derive an explicit description
of the exceptional coupling coefficients.
Finally, we conclude this section with a few examples
of global descriptions of
$\badcoupling(G)$
in \Cref{sec: global balancing conditions}.

\subsection{Circulation forms of facial systems}\label{sec: circulation form}

To connect the abstract algebraic condition in \Cref{def: global exceptional}
to the concrete graph-theoretic features,
we transform
facial systems of $\raks_G$
into equivalent ``circulation forms",
which is understood more easily from the graph-theoretical viewpoint.

\begin{lemma}\label{lem: face system}
    Let $F \ne \varnothing$ be a proper face of $\radjp_G$,
    $\rinc(\dig{G}_F)$ and $\rowv{a}(\dig{G}_F)$
    be the corresponding reduced incidence matrix of $\dig{G}_F$ and its coupling vector, respectively.
    Then $\init_F(\raks_G)(\rowv{x}) = \colv{0}$ 
    if and only if
    \begin{equation}\label{equ: cycle form}
        \rinc(\dig{G}_F)
        \left( 
            \rowv{x}^{ \rinc(\dig{G}_F) }
            \circ
            \rowv{a}(\dig{G}_F)
        \right)^\top
        =
        \colv{0}.
    \end{equation}
\end{lemma}

\begin{proof}
    Recall that $\raks_G = R \cdot \aks_G$
    for a nonsingular matrix $R = [r_{ij}]$
    Since $a_{ij} = a_{ji}$, we have
    \begin{align*}
        \init_F (\rake_{G,k})(\rowv{x}) 
        &= 
        \sum_{(i,j) \in \dig{G}_F} 
            (r_{ki} - r_{kj}) \, 
            a_{ij} \, \rowv{x}^{\colv{e}_i - \colv{e}_j}
        = 
        \sum_{(i,j) \in \dig{G}_F} 
            \inner{ \rowv{r}_k }{ \colv{e}_i - \colv{e}_j } \,
            a_{ij} \, \rowv{x}^{\colv{e}_i - \colv{e}_j},
        \quad\text{for } k = 1,\dots,n,
    \end{align*}
    where $\rowv{r}_k$ is the $k$-th row of the matrix $R$.
    Since $\{ \colv{e}_i - \colv{e}_j \mid  (i,j) \in \edges(\dig{G}_F)$
    are the columns in $\rinc(\dig{G}_F)$,
    \[
        \init_F (\raks_G)(\rowv{x}) 
        = 
        R \, \rinc(\dig{G}_F) \,
        \left( 
            \rowv{x}^{ \rinc(\dig{G}_F) }
            \circ
            \rowv{a}(\dig{G}_F)
        \right)^\top.
    \]
    Since the square matrix $R$ is assumed to be nonsingular,
    this establishes the equivalence.
\end{proof}

Equation \eqref{equ: cycle form} will be referred to
as the \emph{circulation form} of the facial system $\init_F(\raks_G)$
since the null space of $\rinc(\dig{G}_F)$
can be interpreted as the \emph{space of circulations} of $\dig{G}_F$.
Circulations on directed graphs are assignments of flows along directed edges
so that the sum of incoming flows is equal to the sum of outgoing flows at each nodes.
Equation \eqref{equ: cycle form} requires
$\rowv{x}^{\rinc(\dig{G}_F)} \circ \rowv{a}(\dig{G}_F)$
to be a circulation on the facial digraph $\dig{G}_F$,
where the flows are allowed to
have complex values.

\begin{figure}
    \centering
    \subfloat[AP derived from $C_4$]{
        \begin{overpic}[width=0.25\textwidth]{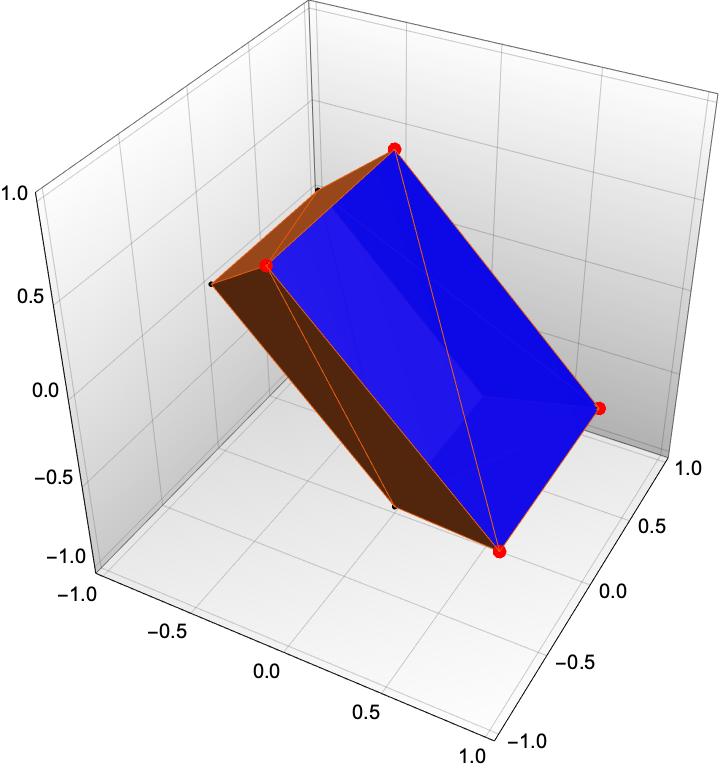}
         \put (20,70) {$\displaystyle e_{32}$}
         \put (40,85) {$\displaystyle e_{30}$}
         \put (60,20) {$\displaystyle e_{12}$}
         \put (75,55) {$\displaystyle e_{10}$}
        \end{overpic}
        \label{fig: AP of C4}
    }
    \subfloat[AP derived from $H_4$]{
        \begin{overpic}[width=0.25\textwidth]{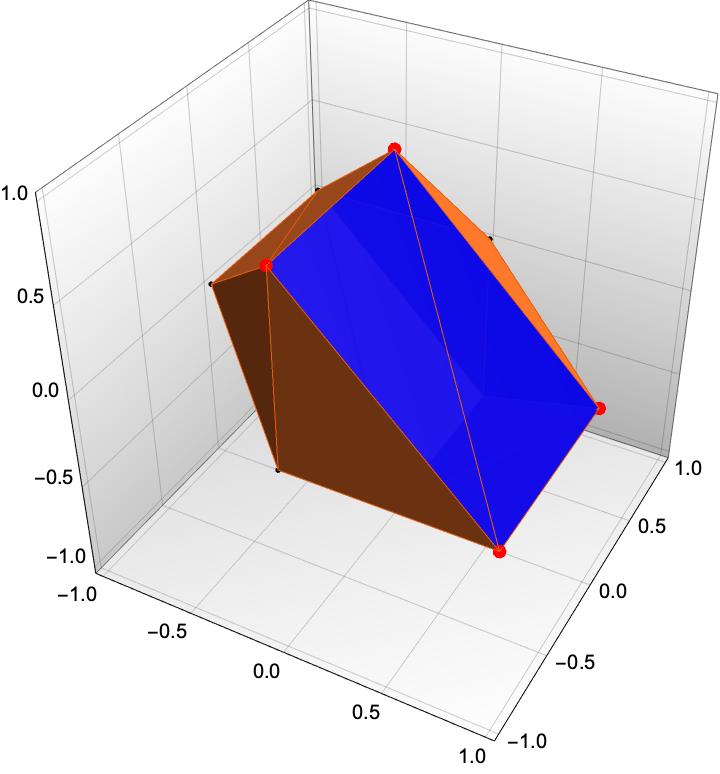}
         \put (20,70) {$\displaystyle e_{32}$}
         \put (40,85) {$\displaystyle e_{30}$}
         \put (60,20) {$\displaystyle e_{12}$}
         \put (75,55) {$\displaystyle e_{10}$}
        \end{overpic}
        \label{fig: AP of H4}
    } 
    \subfloat[AP derived from $K_4$]{
        \begin{overpic}[width=0.25\textwidth]{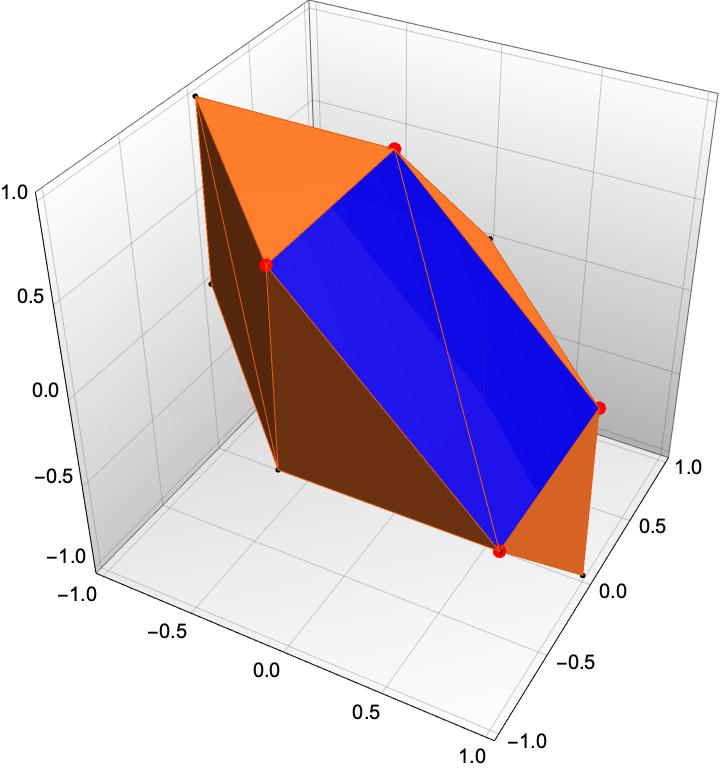}
         \put (15,70) {$\displaystyle e_{32}$}
         \put (40,85) {$\displaystyle e_{30}$}
         \put (60,20) {$\displaystyle e_{12}$}
         \put (75,55) {$\displaystyle e_{10}$}
        \end{overpic}
        \label{fig: AP of K4}
    }
    \subfloat[Facet digraph]{
        \begin{tikzpicture}[
            main/.style = {draw,circle},
            every edge/.style = {draw,-latex,thick},
            scale=0.9
        ] 
            \node[main] (0) at ( 0, 1.8) {$0$}; 
            \node[main] (2) at ( 0,-1.8) {$2$}; 
            \node[main] (1) at ( 1.0, 0) {$1$}; 
            \node[main] (3) at (-1.0, 0) {$3$}; 
            \path (1) edge (0);
            \path (1) edge (2);
            \path (3) edge (0);
            \path (3) edge (2);
        \end{tikzpicture} 
        \label{fig: C4 facet digraph}
    }
    \caption{
        The adjacency polytopes of $C_4,H_4 = K_4 - \{e_{13} \}$ and $K_4$ with the vertex labels $e_{ij} := e_i - e_j$
        together with the facet subdigraph corresponds to the common facet in blue.
    }
    \label{fig: AP examples}
\end{figure}
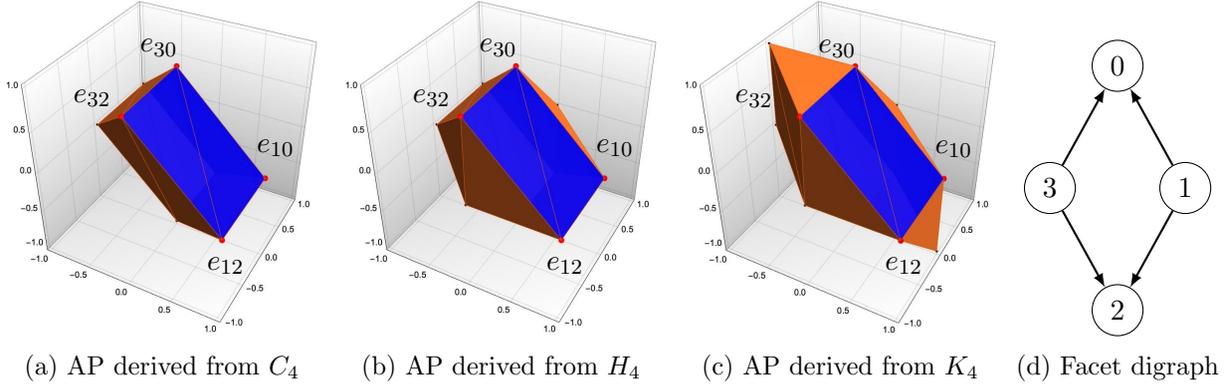

\begin{example}\label{ex: a facet of N=4}
Consider the facial subdigraph in \Cref{fig: C4 facet digraph},
which is shared by $C_4$, $H_4 = K_4 - \{e_{13}\}$, 
and $K_4$.
The corresponding incidence and reduced incidence matrices,
respectively, are
\begin{align*}
    Q(\dig{G}_F) =
    \bordermatrix{ & e_{10} & e_{12} & e_{32} & e_{30} \cr
    0 & -1 & 0 & 0 & -1 \cr
    1 & 1 &  1 &  0 & 0 \cr
    2 &  0 & -1 & -1 & 0 \cr
    3 &  0 &  0 &  1 & 1}, \qquad \qquad
    \rinc(\dig{G}_F) =
    \bordermatrix{ & e_{10} & e_{12} & e_{32} & e_{30} \cr
    1 & 1 &  1 &  0 & 0 \cr
    2 &  0 & -1 & -1 & 0 \cr
    3 &  0 &  0 &  1 & 1}.
\end{align*}
By \Cref{lem: face system}, the facial system
$\init_F(\raks_{C_4})(\rowv{x}) = \colv{0}$ is equivalent to
its circulation form

\begin{align}\label{eq circ form}
       \begin{bmatrix}
        0 \\ 0 \\ 0
    \end{bmatrix} &= \begin{bmatrix}
        1 &  1 &  0 & 0 \\
        0 & -1 & -1 & 0 \\
        0 &  0 &  1 & 1
    \end{bmatrix}
    \left(
        \begin{bmatrix}
            x_1 \; x_2 \; x_3
        \end{bmatrix}^{
            \left[
            \begin{smallmatrix}
                1 &  1 &  0 & 0 \\
                0 & -1 & -1 & 0 \\
                0 &  0 &  1 & 1
            \end{smallmatrix}
            \right]
        }
        \circ
        \begin{bmatrix}
            a_{10} \; a_{12} \; a_{32} \; a_{30}
        \end{bmatrix}
    \right)^\top
    = \begin{bmatrix}
        a_{10}x_1 + a_{12} \frac{x_1}{x_2} \\
        - a_{12} \frac{x_1}{x_2} - a_{32} \frac{x_3}{x_2} \\
        a_{32} \frac{x_3}{x_2} + a_{30} x_3
    \end{bmatrix}
\end{align}

Observe that adding the three equations in \eqref{eq circ form} together
implies that at node $0$ the balancing condition
$-a_{10} x_1 - a_{30} x_3 = 0$ is also satisfied.
Therefore,
$\begin{bmatrix}
    a_{10}x_1 & a_{12}\frac{x_1}{x_2} & a_{32} \frac{x_3}{x_2} & a_{30}x_3
\end{bmatrix}^\top$
is a nonzero circulation on $\dig{G}_F$ if and only if \eqref{eq circ form} is true.

\end{example}

This circulation form suggests a strong tie between
the topological features of a facial subdigraphs $\dig{G}_F$ and
the algebraic features of its facial system $\init_F(\raks_G)$.
The rest of this paper is devoted to exploring this connection
by leveraging the wealth of existing knowledge about
the facial complex of $\radjp_G$.
In particular, 
we will show that the circulation form of facial systems
produce concrete genericity conditions on the coupling coefficients.
Before doing so, we state two immediate observations that simplify our discussions
but defer the proofs to \Cref{sec: lemmas}.

\begin{restatable}{lemma}{componentwise}\label{lem: componentwise facial system}
    For a proper face $F \ne \varnothing$ of $\radjp_G$
    whose facial subdigraph $\dig{G}_F$ consists of
    weakly connected components $\dig{H}_1,\ldots,\dig{H}_\ell$,
    there exist faces $F_1,\ldots,F_\ell \ne \varnothing$ of $F$
    such that $\dig{H}_i = \dig{G}_{F_i}$ and 
    \[
        \init_F(\raks_G)(\rowv{x}) = \colv{0}
        \quad\Longleftrightarrow\quad
        \init_{F_i}(\raks_G)(\rowv{x}) = \colv{0}
        \quad\text{for each } i = 1,\ldots,\ell.
    \]
\end{restatable}

\begin{restatable}{lemma}{smallergraph}\label{lem: smaller sub graph}
    For a proper face $F \ne \varnothing$ of $\radjp_G$,
    let $G' = G[G_F]$ 
    and $F'$ be the embedding of $F$ in $\radjp_{G'}$.
    Then
    $\init_{F }(\raks_{G })$ has a $\CC^*$-zero if and only
    $\init_{F'}(\raks_{G'})$ has a $\CC^*$-zero.
\end{restatable}

Together, these lemmas tell us that it is sufficient to focus on
facial systems corresponding to facial subgraphs that are
connected and spanning, which are exactly the facet systems.

\subsection{Tree solutions}\label{sec: tree solution}

With our focus fixed on facet systems,
we now give an explicit parameterization for
the possible solution sets to a given facet system.

\begin{definition}[Tree solutions]\label{def: tree solution}
    For a given facet $F$ of $\radjp_G$, 
    let $\dig{H} = \dig{G}_F$ be its facet subdigraph.
    We fix a weak spanning tree $\dig{T}$,
    i.e., a subdigraph $\dig{T} \le \dig{H}$
    whose underlying undirected graph is a spanning tree
    of the underlying undirected graph of $\dig{H}$.
    If $\dig{T} \ne \dig{H}$,
    let $\colv{v}_1, \ldots, \colv{v}_d$ be the columns of $\rinc(\dig{H})$
    corresponding to the directed edges in $\edges(\dig{H}) \setminus \edges(\dig{T})$,
    as column vectors,
    and let
    $\rowv{\tau}_k =  (\rinc(\dig{T})^{-1} \colv{v}_k)^\top$
    for $k=1,\ldots,d$.  
    We define the \term{tree solution} induced by $\dig{T}$ to be 
    \begin{equation}\label{equ: tree solution}
        \rowv{x}_{\dig{T}}(\lambda_1,\ldots,\lambda_d) =
        \left[
            \left( \sum_{k=1}^{d} \lambda_k \rowv{\tau}_k \right) \circ
            \rowv{a}_{\dig{T}}^{-I}
        \right]^{
            \rinc(\dig{T})^{-1}
        }.
    \end{equation}
\end{definition}

\begin{example}[A tree solution for $C_4$]
    Consider the graph $G = C_4$
    and the subdigraph $\dig{H}$ shown in \Cref{fig: C4 facet digraph}.
    With the choice $\dig{T} < \dig{H}$ such that
    $\edges(\dig{T}) = \{(1,0),(1,2),(3,2) \}$,
    we have one directed edge in $\dig{H}$
    that is outside $\dig{T}$,
    and it corresponds to the incidence vector
    $\colv{v}_1 = \begin{bmatrix} 0 & 0 & 1 \end{bmatrix}^\top$.
    Then
    \begin{equation*}
        \rowv{\tau}_1 =  (\rinc(\dig{T})^{-1} \colv{v}_1)^\top \\
        = \left( \begin{bmatrix}
            1 &  1 &  0 \\
            0 & -1 & -1 \\
            0 &  0 &  1 \\
        \end{bmatrix}^{-1} \cdot \begin{bmatrix}
            0 \\ 0  \\ 1
        \end{bmatrix} \right)^\top \\
        = \begin{bmatrix}
            1 & -1 & 1
        \end{bmatrix}.
    \end{equation*}
    Therefore, the tree solution induced by $\dig{T}$ is 
    \begin{align*}
        x_{\dig{T}}(\lambda_1)
        &= \left(
            \lambda_1
            \begin{bmatrix}
                1 & -1 & 1
            \end{bmatrix}
            \circ
            \begin{bmatrix}
                a_{10}^{-1} & a_{12}^{-1} & a_{32}^{-1}
            \end{bmatrix}
        \right)^{\left[\begin{smallmatrix}
            1 &  1 &  1 \\
            0 & -1 & -1 \\
            0 &  0 &  1
        \end{smallmatrix}\right]} \\
        &= \begin{bmatrix}
            \lambda_1 a_{10}^{-1} & \frac{\lambda_1 a_{10}^{-1}}{-\lambda_1 a_{12}^{-1}} & \frac{\lambda_1 a_{10}^{-1} \lambda_1 a_{32}^{-1}}{-\lambda_1 a_{12}^{-1}} \end{bmatrix} \\
            &= \begin{bmatrix}
                \frac{\lambda_1}{a_{10}} & \frac{a_{12}}{a_{10}} & \frac{-\lambda_1 a_{12}}{a_{10} a_{32}}
        \end{bmatrix}.
    \end{align*}
\end{example}

Indeed, each tree solution provides a candidate solution 
in the sense that if $\init_F(\raks_{G})$ has a $\CC^*$-zero,
then it must be given by $\rowv{x}_{\dig{T}}$
for any weak spanning tree $\dig{T}$.

\begin{lemma}\label{lem: initial soln must be tree soln}
    If $\rowv{x}$ is a $\CC^*$-zero of $\init_F(\raks_G)$,
    then for any weak spanning tree $\dig{T}$ of $\dig{G}_F$,
    there exists $\lambda_1,\ldots,\lambda_d \in \CC$ such that
    \[
        \rowv{x} = \rowv{x}_{\dig{T}}(\lambda_1,\ldots,\lambda_d).
    \]
\end{lemma}

\begin{proof}
    Let $\dig{H}$, $\dig{T}$, $\rowv{\tau}_1,\ldots,\rowv{\tau}_d$,
    $\colv{v}_1,\ldots,\colv{v}_d$, and $\rowv{x}_{\dig{T}}$
    be those referenced in \Cref{def: tree solution}.
    Suppose $\rowv{x}$ is a $\CC^*$-zero of $\init_F(\raks_G)$.
    We shall arrange the columns/entries, so that
    \begin{align*}
        \rinc(\dig{H}) &= 
        \begin{bmatrix} \rinc(\dig{T}) & \colv{v}_1 & \cdots & \colv{v}_d \end{bmatrix} 
        &&\text{and} &
        \rowv{a}(\dig{H}) &=
        \begin{bmatrix} \rowv{a}(\dig T) & \alpha_1 & \cdots & \alpha_d \end{bmatrix}
    \end{align*}
    where $\alpha_1,\ldots,\alpha_d$ are the complex coupling coefficients
    on the directed edges corresponds to $\colv{v}_1,\ldots,\colv{v}_d$.
    
    By \Cref{lem: face system}, the circulation form of $\init_F(\raks_G)$ is
    \begin{equation}\label{equ: tree solution plugged in}
        \rinc(\dig{H}) \left(
            \rowv{x}^{ \rinc(\dig{H}) }
            \circ
            \rowv{a}(\dig{H})
        \right)^\top
        =
        \begin{bmatrix} \rinc(\dig{T}) & \colv{v}_1 & \cdots & \colv{v}_d \end{bmatrix} 
        \begin{bmatrix}
            \left( \rowv{x}^{\rinc(\dig{T}) } \circ \rowv{a}(\dig{T}) \right)^\top
            \\
            \rowv{x}^{\colv{v}_1} \cdot \alpha_1 \\
            \vdots \\
            \rowv{x}^{\colv{v}_d} \cdot \alpha_d \\
        \end{bmatrix}
        =
        \colv{0}.
    \end{equation}
    Since the null space of $\rinc(\dig{H})$ is spanned by
    $(\rowv{\tau}_1,-\rowv{e}_1)^\top, \ldots, (\rowv{\tau}_d,-\rowv{e}_d)^\top \in \ZZ^{n+d}$.
    There exist $\lambda_1,\ldots,\lambda_d \in \CC$ such that
   \begin{align}
        \sum_{k=1}^d \lambda_k \rowv{\tau}_k =
        \rowv{x}^{\rinc(\dig{T}) } \circ \rowv{a}(\dig{T}). \label{eq : null space equivalence}
    \end{align}
    The conclusion thus follow from the fact that
    the reduced incidence matrix $\rinc(\dig{T})$ is unimodular.
\end{proof}

\Cref{lem: initial soln must be tree soln} shows that any $\CC^*$-zero of a facet system must be given by a tree solution.
Indeed, \eqref{equ: tree solution plugged in}
shows that with an additional condition,
the converse is also true.

\begin{restatable}{corollary}{TreeCondition}\label{cor: tree solution condition}
    Let $F$, $\dig{T}$, $\rowv{x}_{\dig{T}}$, and
    $\rowv{\tau}_1,\ldots,\rowv{\tau}_d$
    be as in \Cref{def: tree solution},
    and let $\rowv{\eta}_i = [\rowv{\tau}_i,-\rowv{e}_i]$ and
    $\colv{\eta}_i = \rowv{\eta}_i^\top$ for $i=1,\ldots,d$.
    Then, for nonzero $\lambda_1,\ldots,\lambda_d$,
    the tree solution $\rowv{x}_{\dig{T}}(\lambda_1,\ldots,\lambda_d)$
    is a solution to the facial system $\init_F (\raks_G)$
    if and only if
    \[
        \left(
            \sum_{k=1}^d \lambda_k \rowv{\eta}_k
        \right)^{\colv{\eta}_i}
        \cdot
        \rowv{a}(\dig{H})^{-\colv{\eta}_i}
        = 1
        \quad\text{for each } i=1,\ldots,d.
    \]
\end{restatable}

An elementary proof is included in the appendix for completeness.
In the following we will provide a graph-theoretic interpretation
of these algebraic conditions.

\subsection{Exceptional coupling characterized by balancing conditions}\label{sec: facial balancing conditions}

Now, we will show that the exceptional coupling coefficients
with respect to a face are exactly those coefficients
that satisfy ``balancing'' conditions on the corresponding face subdigraph.

Several distinct concepts that capture intricate
balancing conditions on Kuramoto networks have been studied.
Following this tradition, we will define ``balanced networks''.
Depending on the additional information we take into consideration,
it has three layers of meanings.

\begin{definition}\label{def: edge-wise balancing}
    We say an acyclic digraph $\dig{H}$,
    whose underlying graph is a cycle,
    is \term{balanced} if it contains
    an equal number of directed edges pointing clockwise and counterclockwise.
\end{definition}

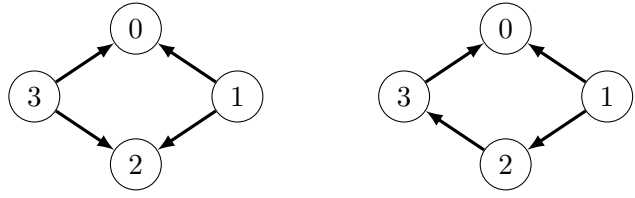
\begin{wrapfigure}[11]{r}{0.6\textwidth}
    \begin{subfigure}[b]{0.29\textwidth}
        \centering
        \begin{tikzpicture}[
            main/.style = {draw,circle},
            every edge/.style = {draw,-latex,very thick},
            scale=0.9
        ] 
            \node[main] (0) at ( 0, 1) {$0$}; 
            \node[main] (2) at ( 0,-1) {$2$}; 
            \node[main] (1) at ( 1.5, 0) {$1$}; 
            \node[main] (3) at (-1.5, 0) {$3$}; 
            \path (1) edge (0);
            \path (1) edge (2);
            \path (3) edge (0);
            \path (3) edge (2);
        \end{tikzpicture} 
        \caption{Balanced: 2 counterclockwise and 2 clockwise edges}
        \label{fig: balanced graph a}
    \end{subfigure} 
    \begin{subfigure}[b]{0.29\textwidth}
        \centering
        \begin{tikzpicture}[
            main/.style = {draw,circle},
            every edge/.style = {draw,-latex,very thick},
            scale=0.9
        ] 
            \node[main] (0) at ( 0, 1) {$0$}; 
            \node[main] (2) at ( 0,-1) {$2$}; 
            \node[main] (1) at ( 1.5, 0) {$1$}; 
            \node[main] (3) at (-1.5, 0) {$3$}; 
            \path (1) edge (0);
            \path (1) edge (2);
            \path (3) edge (0);
            \path (2) edge (3);
        \end{tikzpicture} 
        \caption{Unbalanced: 1 counterclockwise and 3 clockwise edges}
    \end{subfigure}
    \caption{Balanced and unbalanced digraphs.}
    \label{fig: balanced and unbalanced C4}
\end{wrapfigure}
This definition implies that a balanced cycle must be an even cycle.
In \cite{ChenDavisKorchevskaia2022Facets},
it was shown that all facial subdigraphs of $C_N$ must be balanced.
\Cref{fig: balanced and unbalanced C4} shows two subdigraphs of $C_4$,
one balanced and one unbalanced.

In addition, if we take into consideration the weights
on the directed edges, i.e. the coupling coefficients in our context,
we can define a stronger notion of balanced cycle.

\begin{definition}\label{def: weight-wise balancing}
    We say an acyclic digraph $\dig{H}$,
    whose underlying graph is $C_N$,
    is \term{balanced with respect to the weights}
    $K = \{ k_{ij} \mid (i,j) \in \edges(\dig{H}) \}$ if
    it is balanced (as in \cref{def: edge-wise balancing}) and
    \begin{equation}\label{equ: balancing of weights}
        \prod_{(i,j) \in \edges^+(\dig{H})} k_{ij} = 
        (-1)^{N/2}
        \prod_{(i,j) \in \edges^-(\dig{H})} k_{ij},
    \end{equation}
    where $\edges^+(\dig{H})$ and $\edges^-(\dig{H})$
    are the sets of edges pointing clockwise and counterclockwise, respectively.
\end{definition}

\begin{example}
    Continuing with the examples on subdigraphs of $C_4$,
    the digraph shown in \Cref{fig: balanced graph a} 
    is balanced with respect to $K = \{ k_{ij} \}$ if and only if
    $k_{10} k_{32} = k_{12} k_{30}$. 
\end{example}

It turns out that 
this balancing condition characterizes
the exceptional coupling coefficients when the facial subgraph $G_F$ is a cycle. 

\begin{proposition}\label{thm: cycle balancing condition}
    Let $F$ be a face of $\radjp_G$
    such that $G_F$ is a cycle.
    Then the set of exceptional coupling coefficients
    with respect to $F$ is 
    \[
        \badcoupling(\dig{G}_F) = \{
            K \mid
            \dig{G}_F
            \text{ is balanced with respect to the weights assigned by $K$}
        \}.
    \]
\end{proposition}

\begin{proof}
    By \Cref{lem: smaller sub graph},
    it is sufficient to assume that $G_F$ spans $G$.
    By \Cref{lem: initial soln must be tree soln},
    it is sufficient to consider the tree solution
    $\rowv{x}_{\dig{T}}(\lambda)$ (\Cref{def: tree solution})
    induced by a weak spanning tree $\dig{T}$ of $\dig{G}_F$.
    
    Let $\rowv{\eta}^\top \in \{+1,-1\}^{n+1}$ be a
    primitive basis vector for $\ker \rinc(\dig{G}_F)$.
    Its entries gives the orientations of the edges in $\dig{G}_F$.
    Since $\dig{G}_F$ must be edge-wise balanced
    (\Cref{def: edge-wise balancing}),
    exactly half of the entries in $\rowv{\eta}$ is $-1$.
    By \Cref{cor: tree solution condition},
    $\rowv{x}_{\dig{T}}$ parameterizes $\CC^*$-zeros of
    $\init_F(\raks_G)$ if and only if
    \[
        1 =
        (\lambda \rowv{\eta})^{\colv{\eta}} \cdot  \rowv{a}(\dig{G}_F)^{-\colv{\eta}}
        = \rowv{\eta}^{\colv{\eta}} \cdot \rowv{a}(\dig{G}_F)^{-\colv{\eta}}
        = (-1)^{N/2} \rowv{a}(\dig{G}_F)^{-\colv{\eta}}
        = (-1)^{N/2} 
        \prod_{(i,j) \in \edges(\dig{G}_F)} k_{ij}^{\sigma(i,j)},
    \]
    where $\sigma(i,j) \in \{ +1, -1 \}$ are the entries in $\rowv{\eta}$
    and they indicate the orientations
    of the directed edge $(i,j)$ in $\dig{G}_F$.
    Rearranging factors here produces the desired equation
    \eqref{equ: balancing of weights}.
\end{proof}

For subdigraphs beyond cycles,
more complicated balancing conditions are required
to characterize the exceptional coupling coefficients.
Indeed, we need all cycles in $G_F$ to be balanced
with respect to a common circulation.
Recall that a circulation on a digraph $\dig{H}$
is an assignment of weights on the edges
$\eta = \{ \eta_{ij} \mid (i,j) \in \edges(\dig{H}) \}$
so that the sum of the incoming and outgoing weights are balanced on each node,
i.e., for each $i=1,\ldots,n$,
\[\sum_{j, (i,j) \in \edges(\dig{H})} \eta_{ij} = \sum_{k,(k,i) \in \edges(\dig{H})} \eta_{ki}.\]
Circulations correspond to a null vector of the reduced incidence matrix $\rinc(\dig{H})$.
As before, we allow complex weights to be used in circulations.

\begin{definition}\label{def: circulation balancing}
    For an acyclic digraph $\dig{H}$,
    let $H$ be its underlying undirected graph,
    let $\{ C_1,\ldots,C_d \}$ be the cycles in $H$,
    and let $\dig{C}_1,\ldots,\dig{C}_d < \dig{H}$ be their corresponding digraphs.
    
    We say $(\dig{H},K)$ is \term{balanced with respect to a circulation}
    $\eta = \{ \eta_{ij} \ne 0 \mid (i,j) \in \edges(\dig{H}) \}$ if
    $\dig{C}_1,\ldots,\dig{C}_d$ are balanced
    (in the sense of \Cref{def: edge-wise balancing}), and
    \begin{equation}\label{equ: balancing of circulation}
        \prod_{(i,j) \in \edges^+(\dig{C}_\ell)} \frac{k_{ij}}{\eta_{ij}} = 
        \prod_{(i,j) \in \edges^-(\dig{C}_\ell)} \frac{k_{ij}}{\eta_{ij}},
    \end{equation}
    for each $\ell=1,\ldots,d$,
    where $\edges^+(\dig{C}_k)$ and $\edges^-(\dig{C}_k)$
    are the sets of directed edges of $\dig{C}_k$
    in the counterclockwise and clockwise orientations, respectively.
\end{definition}

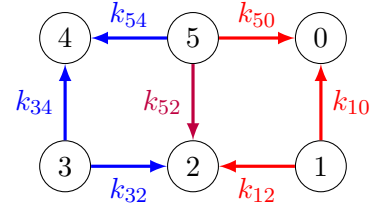
\begin{wrapfigure}[10]{r}{0.30\textwidth}
    \centering
    \begin{tikzpicture}[
        main/.style = {draw,circle},
        every edge/.style = {draw,-latex,very thick},
        scale=1.7
    ] 
        \node[main] (0) at ( 2, 1) {$0$}; 
        \node[main] (1) at ( 2, 0) {$1$}; 
        \node[main] (2) at ( 1, 0) {$2$}; 
        \node[main] (3) at ( 0, 0) {$3$}; 
        \node[main] (4) at ( 0, 1) {$4$}; 
        \node[main] (5) at ( 1, 1) {$5$}; 
        \path[color=red] (1) edge node[right] {$k_{10}$}  (0);
        \path[color=red] (1) edge node[below] {$k_{12}$}  (2);
        \path[color=blue] (3) edge node[left ] {$k_{34}$}  (4);
        \path[color=blue] (3) edge node[below] {$k_{32}$}  (2);
        \path[color=blue] (5) edge node[above] {$k_{54}$}  (4);
        \path[color=purple] (5) edge node[left ] {$k_{52}$}  (2);
        \path[color=red] (5) edge node[above] {$k_{50}$}  (0);
    \end{tikzpicture} 
    \caption{Balancing condition with respects to weights}
    \label{fig: figure 8 weight balancing}
\end{wrapfigure}
Here, it is sufficient to restrict $\{ C_1, \ldots, C_d \}$
to be a cycle basis of $H$.
Also note that this definition takes into consideration both
the weights on directed edges and a given circulation.
It is a generalization of
\Cref{def: weight-wise balancing}
since the vector space of circulations on a cycle is 1-dimensional.

\begin{example}
    \Cref{fig: figure 8 weight balancing}
    shows a digraph with weighted edges.
    It is balanced with respect to a circulation
    $\eta = \{ \eta_{ij} \}$ if
    \begin{align}\label{equ: figure 8 weight balancing 1}
        \frac{ \textcolor{red}{k_{50}} \textcolor{red}{   k_{12}} }{ \eta_{50} \eta_{12} } &= 
        \frac{ \textcolor{red}{k_{10}} \textcolor{purple}{k_{52}} }{ \eta_{10} \eta_{52} }
        &&\text{and} &
        \frac{ \textcolor{purple}{k_{52}} \textcolor{blue}{k_{34}} }{ \eta_{52} \eta_{34} } &= 
        \frac{ \textcolor{blue}{  k_{54}} \textcolor{blue}{k_{32}} }{ \eta_{54} \eta_{32} }.
    \end{align}
\end{example}

This stronger balancing condition gives us the exact description
for the exceptional coupling coefficients
in the bridgeless cases.
Recall, a \emph{bridge} of a graph is an edge
that is not contained in any cycle of the graph. 
A graph with no bridges is said to be \emph{bridgeless}.

\begin{theorem}\label{thm: circulation condition}
    Let $F$ be a face of $\radjp_G$.
    \begin{itemize}
        \item
            If $G_F$ contains a bridge,
            then $\badcoupling(\dig{G}_F) = \varnothing$.
        \item
            If $G_F$ is bridgeless,
            then
            \[
                \badcoupling(\dig{G}_F) = \{
                    K \mid
                    \dig{G}_F \text{ is balanced with respect to
                    a nonzero circulation on } \dig{G}_F
                \}.
            \]
    \end{itemize}
\end{theorem}

\begin{proof}
    We now consider the first case. 
    If $G_F$ contains a bridge,
    then any null vector of $\rinc(\dig{G}_F)$
    must contain a zero entry,
    which corresponds to the bridge edge.
    By \Cref{lem: initial soln must be tree soln},
    the facial system $\init_F(\raks_G)$ has no $\CC^*$-zero
    for any nonzero choices of coefficients.
    
    For the second case, as remarked earlier,
    we can assume $G_F$ is both connected and spanning.
    Fixing any weak spanning tree $\dig{T}$ for $\dig{G}_F$,
    we consider the tree solution $\rowv{x}_{\dig{T}}$
    (\Cref{def: tree solution}).
    Let $\rowv{\tau}_1,\ldots,\rowv{\tau}_d$ and
    $\rowv{\eta}_\ell = [\rowv{\tau}_\ell, - \rowv{e}_\ell]$ 
    be those referenced in \Cref{cor: tree solution condition}.
    
    Suppose $(\dig{G}_F,K)$ is balanced with respect to a
    nonzero circulation $\eta = \{ \eta_{ij} \}$.
    Then 
    \begin{align*}
        \prod_{(i,j) \in \edges^+(\dig{C}_\ell)} \frac{k_{ij}}{\eta_{ij}} &= 
        \prod_{(i,j) \in \edges^-(\dig{C}_\ell)} \frac{k_{ij}}{\eta_{ij}} &
        &\text{i.e.,} &
        1 &=
        \prod_{(i,j) \in \edges(\dig{C}_\ell)} 
        \left(
            \frac{k_{ij}}{\eta_{ij}}
        \right)^{\sigma_{ij}}
    \end{align*}
    for each $\ell = 1,\ldots,d$,
    where $\sigma_{ij} \in \{ +1, -1 \}$ indicate the orientation of
    the edge $(i,j)$ in $\dig{C}_\ell$.

    Recall that $\rowv{\eta}_1,\ldots,\rowv{\eta}_d$ spans
    the vector space of circulations.
    So after arranging $\{ \eta_{ij} \}$ into a row vector $\rowv{\eta}$
    according to the edge ordering that is consistent with
    $\rowv{\eta}_1,\ldots,\rowv{\eta}_d$,
    there exist nonzero $\lambda_1,\dots,\lambda_d$ such that
    $
        \rowv{\eta} = \lambda_1 \rowv{\eta}_1 + \cdots \lambda_d \rowv{\eta}_d.
    $
    Combining this and the previous equation, we get
    \[
        1 =
        \prod_{(i,j) \in \edges(\dig{C}_\ell)} 
        \left(
            \frac{k_{ij}}{\eta_{ij}}
        \right)^{\sigma_{ij}}
        =
        \left(
            \lambda_1 \rowv{\eta}_1 + \cdots \lambda_d \rowv{\eta}_d
        \right)^{\colv{\eta}_\ell}
        \rowv{a}(\dig{G}_F)^{-\colv{\eta}}
    \]
    for $\ell = 1,\ldots,d$.
    By \Cref{cor: tree solution condition},
    $\init_F(\raks_{(G,K)})$ has a $\CC^*$-zero,
    and thus the coupling coefficients given by $K$
    is among the exceptional coupling coefficients $\badcoupling(\dig{G}_F)$.

    Conversely, if $K \in \badcoupling(\dig{G}_F)$,
    then, by \Cref{lem: initial soln must be tree soln},
    a $\CC^*$-zeros of $\init_F(\raks_{(G,K)})$
    is given by the tree solution
    $\rowv{x}_{\dig{T}}(\lambda_1,\ldots,\lambda_d)$
    (\Cref{def: tree solution}) for some $\lambda_1,\ldots,\lambda_d \ne 0$.
    Moreover, by \Cref{cor: tree solution condition},
    \[
        1 =
        \left(
            \lambda_1 \rowv{\eta}_1 + \cdots \lambda_d \rowv{\eta}_d
        \right)^{\colv{\eta}_\ell}
        \rowv{a}(\dig{G}_F)^{-\colv{\eta}}
        =
        \prod_{(i,j) \in \edges(\dig{C}_\ell)} 
        \left(
            \frac{k_{ij}}{\eta_{ij}}
        \right)^{\sigma_{ij}},
    \]
    Since $\rowv{x}_{\dig{T}}(\lambda_1,\ldots,\lambda_d) \in (\CC^*)^n$,
    $\lambda_1,\ldots,\lambda_d$ are nonzero,
    and $\rowv{\eta} = \lambda_1 \rowv{\eta}_1 + \cdots \lambda_d \rowv{\eta}_d$
    is a nonzero circulation on $\dig{G}_F$.
    Therefore, $(\dig{G}_F,K)$ is balanced with respect to the
    nonzero circulation $\rowv{\eta}$.
\end{proof}

Although \Cref{thm: circulation condition}
depends on the existence of a circulation,
as \Cref{thm: cycle balancing condition} has shown,
this balancing condition can be simplified
in the case where $G_F$ is a cycle.
We now consider one more case in which a similar simplification can be achieved.

\begin{proposition}\label{prop: shared edge}
    If a facial subdigraph $\dig{G}_F$
    is the graph union of $\dig{C}_1,\ldots,\dig{C}_d$
    whose underlying undirected graphs $C_1,\ldots,C_d$
    are independent cycles of $G_F$ sharing a single edge,
    then $\badcoupling(\dig{G}_F)$ consists of exactly the 
    coupling coefficients for which
    \begin{equation}
        1 = 
        \sum_{\ell = 1}^d
        (-1)^{|C_\ell| / 2}
        \frac{
            \prod_{ (i,j) \in \edges^+(\dig{C}_\ell) } k_{ij}
        }{
            \prod_{ (i,j) \in \edges^-(\dig{C}_\ell) } k_{ij}
        },
    \end{equation}
    where $\edges^+(\dig{C}_\ell)$ and $\edges^-(\dig{C}_\ell)$
    are the set of directed edges of $\dig{C}_\ell$
    in the two possible orientations, respectively,
    and the edge shared by $C_1,\ldots,C_d$ is considered negatively oriented.
\end{proposition}

\begin{proof}
    Recall that the face subgraph $G_F$ must be bipartite
    (\Cref{thm: faces are max bipartite})
    and therefore the cycles $C_1,\ldots,C_d$ must be even.
    
    Suppose $K = \{ k_{ij} \}$ is in $\badcoupling(\dig{G}_F)$,
    then by \Cref{thm: circulation condition}, 
    there exists a circulation $\eta = \{ \eta_{ij} \ne 0 \}$ on $\dig{G}_F$
    satisfying \eqref{equ: balancing of circulation}.
    Let $m_\ell = |C_\ell| - 1$,
    then the space of circulations is spanned by vectors
    \[
        \rowv{\eta}_i \;= 
        \begin{tikzpicture}[decoration=brace,baseline={([yshift=-2ex]current bounding box.center)}]
            \matrix (m) [matrix of math nodes,left delimiter=[,right delimiter={]}] {
                0 & \cdots & 0 &
                \rowv{\tau}_\ell &
                0 & \cdots & 0 &
                -1 \\
            };
            \draw[decorate,thick]
            (m-1-1.north west) -- node[above=2pt] {\scriptsize $\sum_{j=1}^{\ell-1} m_j$} (m-1-3.north east);
            \draw[decorate,thick]
            (m-1-5.north west) -- node[above=2pt] {\scriptsize $\sum_{j=\ell+1}^d m_j$} (m-1-7.north east);
        \end{tikzpicture}
        \hspace{8ex} \text{for } \ell = 1,\ldots,d
    \]
    corresponding to the cycles $C_1,\ldots,C_d$.
    Therefore, when viewed as a vector, $\eta$ can be expressed as
    \[
        \rowv{\eta} = 
        \sum_{i=1}^d \lambda_i \rowv{\eta}_i = 
        \begin{bmatrix}
            \lambda_1 \rowv{\tau}_1 &
            \cdots &
            \lambda_d \rowv{\tau}_d &
            -(\lambda_1 + \cdots + \lambda_d)
        \end{bmatrix}
    \]
    for some $\lambda_i \in \CC^*$ such that $\lambda_1 + \cdots + \lambda_d \ne 0$.
    We can verify that, along each cycle $C_\ell$, we have
    \begin{align*}
        \frac{
            \prod_{(i,j) \in \edges^+(\dig{C}_\ell)} k_{ij}
        }{
            \prod_{(i,j) \in \edges^+(\dig{C}_\ell)} \eta_{ij}
        }
        &=
        \frac{
            \prod_{(i,j) \in \edges^+(\dig{C}_\ell)} k_{ij}
        }{
            \lambda_\ell^{(m_\ell+1)/2}
        }
        &&\text{and} &
        \frac{
            \prod_{(i,j) \in \edges^-(\dig{C}_\ell)} k_{ij}
        }{
            \prod_{(i,j) \in \edges^-(\dig{C}_\ell)} \eta_{ij}
        }
        =
        \frac{
            \prod_{(i,j) \in \edges^-(\dig{C}_\ell)} k_{ij}
        }{
            (-1)^{\frac{m_\ell+1}{2}}
            \lambda_\ell^{\frac{m_\ell-1}{2}}
            (\lambda_1 + \dots + \lambda_d)
        }
    \end{align*}
    the equality between the two
    (equation \eqref{equ: balancing of circulation})
    thus implies that
    \[
        (-1)^{(m_\ell+1)/2}
        \frac{
            \prod_{(i,j) \in \edges^+(\dig{C}_\ell)} k_{ij}
        }{
            \prod_{(i,j) \in \edges^-(\dig{C}_\ell)} k_{ij}
        }
        =
        \frac{
            \lambda_\ell
        }{
            (\lambda_1 + \dots + \lambda_d)
        }
        \quad\text{for } \ell = 1,\ldots,d.
    \]
    Summing over $\ell = 1,\ldots,d$ thus produce the desired equality.

To prove the converse, we adopt the notation
    \begin{align*}
        P_\ell &= \prod_{(i,j) \in \edges^+(\dig{C}_\ell)} k_{ij} &
        N_\ell &= \prod_{(i,j) \in \edges^-(\dig{C}_\ell)} k_{ij} &
        t_\ell &= (-1)^{|C_\ell| / 2} \frac{P_\ell}{N_\ell}
    \end{align*}
    for each $\ell = 1,\ldots,d$.
    Assuming
    $\sum_{\ell =1}^d (-1)^{|C_\ell| \slash 2} \frac{P_\ell}{N_\ell} = t_1 + \cdots + t_d =1$,
    we want to show that
    there is a nonzero circulation $\mathbf{\eta}$ on $\dig{G}$
    that satisfies \eqref{equ: balancing of circulation}.
    Let $e \in \edges(\dig{G})$ be the shared edge,
    which is the only edge shared by the cycles,
    so we have a well-defined assignment $\eta = \{ \eta_{ij} \}$,
    given by
    \[
        \eta_{ij} = \begin{cases}
            t_\ell &\text{if } e_{ij} \in \edges^+(\dig{C}_\ell) \\
            -t_\ell &\text{if } e_{ij} \in \edges^-(\dig{C}_\ell) \setminus \{e\} \\
            -1 &\text{if } (i,j) = e
        \end{cases}
    \]
    which, by assumption, gives a valid nonzero circulation on $\dig{G}$.
    Thus, we can verify \eqref{equ: balancing of circulation} via
    \[
        \prod_{(i,j) \in \edges^+(\dig{C}_\ell)} \frac{k_{ij}}{\eta_{ij}}
        =
        \frac{P_\ell}{t_\ell^{|C_\ell|/2}}
        =
        \frac{P_\ell}{(-1)^{|C_\ell| / 2} \frac{P_\ell}{N_\ell} t_\ell^{|C_\ell|/2-1}}
        =
        \frac{N_\ell}{(-1) (-t_\ell)^{|C_\ell|/2-1}}
        =
        \prod_{(i,j) \in \edges^-(\dig{C}_\ell)} \frac{k_{ij}}{\eta_{ij}}.
    \]
\end{proof}

\begin{example}[Simplified balancing condition for \Cref{fig: figure 8 weight balancing}]\label{ex: 2 4 cycles}
    Through this formulation, the balancing condition
    given in \eqref{equ: figure 8 weight balancing 1}
    for the digraph shown in
    \Cref{fig: figure 8 weight balancing} can be simplified into
    \[
        \frac{ {\color{red} k_{50}} {\color{red} k_{12}} }{ {\color{purple}k_{52}} {\color{red} k_{10}} } +
        \frac{ {\color{blue}k_{54}} {\color{blue}k_{32}} }{ {\color{purple}k_{52}} {\color{blue}k_{34}} } = 1.
    \]
\end{example}

\subsection{Examples of global descriptions of exceptional coupling coefficients}
\label{sec: global balancing conditions}

The previous section gave an explicit description of
the set of exceptional coupling coefficients
with respect to a given face $F$ of $\radjp_G$.
The global description of the exceptional coupling coefficients
$\badcoupling(G)$ can therefore be obtained
as the union of these local contributions from each face.
We conclude this section with concrete descriptions of $\badcoupling(G)$ for families of graphs.
First, note that many faces of $\radjp_G$ have no contribution to
$\badcoupling(G)$.
Combining \Cref{thm: faces are max bipartite} and \Cref{thm: circulation condition},
we get the following topological constraints on the faces with
potentially nontrivial contribution to $\badcoupling(G)$.

\begin{proposition}\label{prop: global exceptional coupling}
    Let $\Phi$ be the set of nontrivial proper faces of $\radjp_G$
    with facial subgraphs that are
    bridgeless and
    maximally bipartite in their induced subgraphs.
    Then
    \[
        \badcoupling(G) = \bigcup_{F \in \Phi} \badcoupling(\dig{G}_F).
    \]
\end{proposition}

From this, we can derive a stronger 
root counting result for networks with no even cycles.

\begin{corollary}
    If $G$ contains no even cycles,
   then $\aks_{(G,K,\colv{w})}$ is Bernshtein-general, and its number of isolated $\CC^*$-zeros,
    counted with multiplicity,
    equals the $\nvol(\radjp_G)$
    for any real or complex natural frequencies $\colv{w}$
    and any real of complex nonzero coupling coefficients
    $K=\{ k_{ij} \}$. 
\end{corollary}

In particular, for odd cycles and trees,
we can use the results from \cite{ChenDavisMehta2018Counting}
to derive an exact root count for \textit{any} Kuramoto network.

\begin{corollary}
    If $G = T_N$ is a tree with $N$ nodes then $\aks_{(G,K,\colv{w})}$ has $2^{N-1}$
    $\CC^*$-zeros, counting multiplicity, for any choices of real or complex natural frequencies $\colv{w}$
    and any choices of real or complex nonzero coupling coefficients
    $K=\{ k_{ij} \}$. 
\end{corollary}

\begin{corollary}
    If $G = C_N$ is an odd cycle then $\aks_{(G,K,\colv{w})}$ has
    $N \binom{N-1}{(N-1) \slash 2}$ $\CC^*$-zeros,
    counting multiplicity,
    for any choices of real or complex natural frequencies $\colv{w}$
    and any choices of real or complex nonzero coupling coefficients
    $K=\{ k_{ij} \}$. 
\end{corollary}

For even cycles, we can explicitly state degeneracy conditions
in terms of balanced networks.

\begin{corollary}\label{thm: unicycle non-bernshtein condition}
    If $G$ is an even cycle, then
    for any real or complex $\colv{w}$,
    the following are equivalent:
    \begin{enumerate}
        \item
            the isolated $\CC^*$-root count of $\aks_{(G,K,\colv{w})}$,
            counted with multiplicity,
            is strictly less than $\nvol(\radjp_G)$;
        \item
            $(G,K)$ has a subnetwork that is balanced
            with respect to weights $K$
            (\Cref{def: weight-wise balancing}).
    \end{enumerate}
\end{corollary}

Note that by \Cref{lem: leaf extension},
any result for cycle graphs extends to \emph{unicycle graphs},
i.e. graphs that contain a unique cycle.

\begin{figure}[h]
    \centering
    \begin{tikzpicture}[
        main/.style = {draw,circle,minimum size=1.5ex},
        every edge/.style = {draw,thick,-latex}
    ] 
        \node[main] (0) at ( 0, 0.8) {$0$}; 
        \node[main] (2) at ( 0,-0.8) {$2$}; 
        \node[main] (1) at ( 1.9, 0) {$1$}; 
        \node[main] (3) at (-1.9, 0) {$3$}; 
        \path (1) edge node[above right] {$k_{10}$} (0);
        \path (1) edge node[below right] {$k_{12}$} (2);
        \path (3) edge node[above left] {$k_{30}$} (0);
        \path (3) edge node[below left] {$k_{32}$} (2);
    \end{tikzpicture} 
    \hspace{3ex}
    \begin{tikzpicture}[
        main/.style = {draw,circle,minimum size=1.5ex},
        every edge/.style = {draw,thick,-latex}
    ] 
        \node[main] (0) at ( 0, 0.8) {$0$}; 
        \node[main] (2) at ( 0,-0.8) {$2$}; 
        \node[main] (1) at ( 1.9, 0) {$1$}; 
        \node[main] (3) at (-1.9, 0) {$3$}; 
        \path (1) edge node[above right] {$k_{10}$} (0);
        \path (1) edge node[below right] {$k_{12}$} (2);
        \path (0) edge node[above left] {$k_{30}$} (3);
        \path (2) edge node[below left] {$k_{32}$} (3);
    \end{tikzpicture}
    \hspace{3ex}
    \begin{tikzpicture}[
        main/.style = {draw,circle,minimum size=1.5ex},
        every edge/.style = {draw,thick,-latex}
    ] 
        \node[main] (0) at ( 0, 0.8) {$0$}; 
        \node[main] (2) at ( 0,-0.8) {$2$}; 
        \node[main] (1) at ( 1.9, 0) {$1$}; 
        \node[main] (3) at (-1.9, 0) {$3$}; 
        \path (0) edge node[above right] {$k_{10}$} (1);
        \path (1) edge node[below right] {$k_{12}$} (2);
        \path (0) edge node[above left] {$k_{30}$} (3);
        \path (3) edge node[below left] {$k_{32}$} (2);
    \end{tikzpicture} 
    \caption{Representatives of the classes of balanced subnetworks in $C_4$}
    \label{fig: C4 classes of facet subgraphs}
\end{figure}
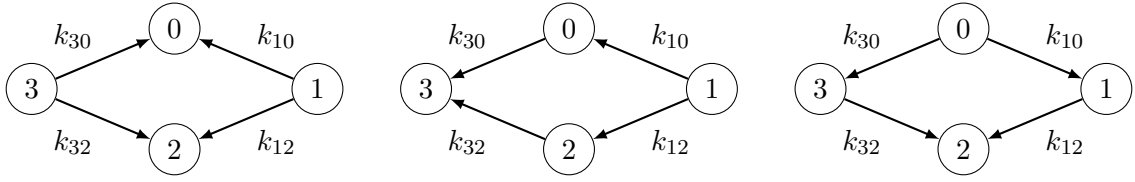
\begin{example}[The 4-cycle case]\label{ex: C4}
    Consider the case of $G = C_4$.
    As shown in \cite[Theorem 16]{ChenDavisMehta2018Counting}
    and \Cref{thm: generic root count},
    for generic choices of $\colv{w}$ and $K$,
    $\aks_G$ has 12 $\CC^*$-zeros.
    This is the adjacency polytope bound.
    With the choice of $\colv{w}$ remain generic,
    we will illustrate the conditions given \Cref{thm: unicycle non-bernshtein condition}.
    $C_4$ has up to 6 balanced subnetworks
    supported by the digraphs shown in \Cref{fig: C4 classes of facet subgraphs}
    and their transposes.
    With these, we have the following decomposition of
    the coupling coefficients:
    \begin{enumerate}[itemsep=1ex]
        \item
            Consider the 1-dimensional ``balancing variety''
            defined by the binomial system
            \begin{equation}\label{equ: C4 balancing variety}
                \left\{
                \begin{aligned}
                    k_{10} k_{12} k_{32}^{-1} k_{30}^{-1} &= 1 \\
                    k_{10} k_{12}^{-1} k_{32}^{-1} k_{30} &= 1 \\
                    k_{10} k_{12}^{-1} k_{32} k_{30}^{-1} &= 1,
                \end{aligned}
                \right.
            \end{equation}
            which will produce all 6 balanced subnetworks.
            It can be parameterized by
            \[
                k_{10} = k_{12} = k_{32} = k_{30} = s,
            \]
            for $s \in \CC^*$.
            Numerical computation shows that
            generic choice of $s \in \CC^*$
            produces an algebraic Kuramoto system $\aks_G$
            with only $6$ $\CC^*$-zeros.
            Moreover, taking $s = 1$
            and
            $w_1 = 1.1, w_2 = - 2.1, w_3 = 1$,
            all six $\CC^*$-zeros can be \emph{real}
            
        \item
            There are three 2-dimensional balancing varieties,
            all containing the 1-dimensional variety defined in
            \eqref{equ: C4 balancing variety},
            which are defined by two out of the three equations
            from this system, i.e.,
            \begin{align*}
                &
                \left\{
                \begin{aligned}
                    k_{10} k_{12} k_{32}^{-1} k_{30}^{-1} &= 1 \\
                    k_{10} k_{12}^{-1} k_{32}^{-1} k_{30} &= 1
                \end{aligned}
                \right.
                &&\text{or} &&
                \left\{
                \begin{aligned}
                    k_{10} k_{12}^{-1} k_{32}^{-1} k_{30} &= 1 \\
                    k_{10} k_{12}^{-1} k_{32} k_{30}^{-1} &= 1,
                \end{aligned}
                \right.
                &&\text{or} &&
                \left\{
                \begin{aligned}
                    k_{10} k_{12} k_{32}^{-1} k_{30}^{-1} &= 1 \\
                    k_{10} k_{12}^{-1} k_{32} k_{30}^{-1} &= 1,
                \end{aligned}
                \right.
            \end{align*}
            each producing 4 balanced subnetworks.
            The first, for example,
            can be parameterized as
            \begin{align*}
                k_{10} &= s &
                k_{12} &= t &
                k_{32} &= \pm t &
                k_{30} &= \pm s &
                \text{for } (s,t) &\in (\CC^*)^2.
            \end{align*}
            The other two can be parameterized similarly.
            Numerical computations suggest that
            generic $(s,t)$ produce algebraic Kuramoto systems
            having 8 solutions.
            Moreover, taking $s = 1, t = -1.001, w_1 = 1.1, w_2 = -2.1, w_3 = 1$ all $8$ solutions can be real. 
            
        \item
            There are three 3-dimensional balancing varieties
            defined by binomial equations
            \begin{align*}
                k_{10} k_{12} k_{32}^{-1} k_{30}^{-1} &= 1 &
                k_{10} k_{12}^{-1} k_{32}^{-1} k_{30} &= 1 &
                k_{10} k_{12}^{-1} k_{32} k_{30}^{-1} &= 1,
            \end{align*}
            each producing 2 balanced subnetworks.
            The first of the three
            can be parameterized by
            \begin{align*}
                k_{10} &= s &
                k_{12} &= u/s &
                k_{32} &= t &
                k_{30} &= u/t ,
            \end{align*}
            for $(s,t,u) \in (\CC^*)^3$.
            The other two can be parameterized similarly.
            Numerical experiments show that 
            a generic choice of coupling coefficients on one of these varieties gives 10 $\CC^*$-zeros.
            Moreover, with $u = 1.01, s = 1, t = - 1.001, w_1 = 1.1, w_2 = -2.1, w_3 = 1$, all zeros are real.
            
        \item
            The remaining choices of $K$,
            form a Zariski-dense subset in the space of all
            coupling coefficients for $C_4$, and
            define an algebraic Kuramoto system with 12 $\CC^*$-zeros.
    \end{enumerate}
    Under the assumption that $\colv{w}$ is generic,
    there are no other possibilities.
    Interestingly, with special choices of $\colv{w}$,
    the same conditions on exceptional coupling coefficients
    will produce positive-dimensional solutions.
    This behavior is explored in \Cref{ex: C4 +dimensional}.
\end{example}

\begin{example}
We now consider a graph $G$ formed as two 4-cycles
sharing one edge as in \Cref{fig: figure 8 weight balancing}.
The geometry of the adjacency polytope for
this and similar graphs was studied in detail by
D'Ali, Delucchi, and Micha\l{}ek \cite{DAliDelucchiMichalek2022Many}
and the related work~\cite{ChenKorchevskaia2019Graph}.
By treating each cycle as a copy of $C_4$,
the equations from \Cref{ex: C4} with respect to the cycles
on vertices $\{v_0,v_1,v_2,v_5\}$ and $\{v_2,v_3,v_4,v_5\}$
apply here as well.
In addition, 9 equations
define other patches of $\badcoupling(G)$:
\begin{align*}
    \frac{ k_{50} k_{12} }{ k_{52} k_{10} } +       
    \frac{ k_{54} k_{32} }{ k_{52} k_{34} } &= 1\,, 
    &
    \frac{ k_{50} k_{12} }{ k_{52} k_{10} } +       
    \frac{ k_{43} k_{32} }{ k_{52} k_{45} } &= 1\,, 
    &
    \frac{ k_{50} k_{12} }{ k_{52} k_{10} } +       
    \frac{ k_{54} k_{43} }{ k_{52} k_{23} } &= 1\,, 
    \\
    \frac{ k_{01} k_{12} }{ k_{52} k_{05} } +       
    \frac{ k_{54} k_{32} }{ k_{52} k_{34} } &= 1\,, 
    &
    \frac{ k_{01} k_{12} }{ k_{52} k_{05} } +       
    \frac{ k_{43} k_{32} }{ k_{52} k_{45} } &= 1\,, 
    &
    \frac{ k_{01} k_{12} }{ k_{52} k_{05} } +       
    \frac{ k_{54} k_{43} }{ k_{52} k_{23} } &= 1\,, 
    \\
    \frac{ k_{01} k_{50} }{ k_{52} k_{12} } +       
    \frac{ k_{54} k_{32} }{ k_{52} k_{34} } &= 1\,, 
    &
    \frac{ k_{01} k_{50} }{ k_{52} k_{12} } +       
    \frac{ k_{43} k_{32} }{ k_{52} k_{45} } &= 1\,, 
    &
    \frac{ k_{01} k_{50} }{ k_{52} k_{12} } +       
    \frac{ k_{54} k_{43} }{ k_{52} k_{23} } &= 1. 
\end{align*}
They corresponds to the 9 transpose-pairs of
maximally bipartite balanced digraphs in $G$, and each defines a subset
of exceptional coupling coefficients of co-dimension 1.

\end{example}

\section{Positive-dimensional solutions for homogeneous networks}\label{sec: +dimensional}

The previous section explored structure of the
exceptional coupling coefficients.
We now turn our attention to the interesting role played by
the natural frequencies $\colv{w}$.
Indeed positive-dimensional zero sets for the algebraic Kuramoto equations
that can appear when both $K$ and $\colv{w}$ are chosen to be special values.
They represent synchronization configurations
that have at least one degree of freedom
(even after fixing a rotational frame).
In this section, we add to the existing literature on constructions of
positive-dimensional zero sets for the Kuramoto equations
\cite{AshwinBickBurylko2016Identical,coss2018locating,harrington2023kuramoto, LindbergZachariahBostonLesieutre2022Distribution,sclosa2022kuramoto}
and characterize conditions under which they arise.
We focus on the most widely studied case of homogeneous oscillators,
i.e., $\colv{w} = w \colv{1}$ for some real or complex constant $w$.
They correspond to Kuramoto systems \eqref{equ: algebraic kuramoto}
with zero constant terms.
Therefore, without loss of generality we will fix $w_i = 0$ for each $i=0,\ldots,n$
in this section and use the notation $(G,K,\colv{0})$
for a Kuramoto network of homogeneous oscillators.

It is worth reiterating that positive-dimensional $\CC^*$-zero sets
for the algebraic Kuramoto system can only occur when
when the vector of coupling coefficients is in
the set of exceptional coupling coefficients
and hence the system is not Bernshtein-general.

\begin{corollary}
    For a connected graph $G$,
    if the $\CC^*$-zero set of the algebraic Kuramoto system $\aks_{(G,K)}$ 
    is positive-dimensional,
    then $K \in \badcoupling(G)$.
\end{corollary}

Under the assumption that the oscillators are homogeneous,
we now study the additional topological and algebraic conditions
that guarantee the existence of positive-dimensional zero sets.

We start our investigation with bipartite networks
for which positive-dimensional solutions can be constructed
as restrictions of tree solutions (\Cref{def: tree solution})
derived from certain facet subdigraphs.
This is because when $G$ is bipartite,
\Cref{thm: faces are max bipartite} gives that $G$ itself is a facet subgraph,
i.e., there exists a $F \in \facets(\radjp_G)$ such that $G = G_F$.
In this case,
$\radjp_G = \{ 0 \} \cup F \cup (-F)$.
That is, $F$ and $-F$ contain the exponent vectors of all terms,
and then the algebraic Kuramoto system can be expressed in the simple form
\begin{equation}\label{equ: bipartite facet split}
    \aks_G(\rowv{x})
    = 
    -
    \rinc(\dig{G}_F) ( \rowv{x}^{ \rinc(\dig{G}_F) } \circ \rowv{a}(G))^\top
    +
    \rinc(\dig{G}_F) ( \rowv{x}^{-\rinc(\dig{G}_F) } \circ \rowv{a}(G))^\top.
\end{equation}
Note that the two terms are the circulation forms of
$\init_F (\raks_G)$ and $\init_{-F} (\raks_G)$, respectively.

In the previous section,
we constructed tree solutions (\Cref{def: tree solution})
for which the first term of \eqref{equ: bipartite facet split} vanishes
under the assumption that $K \in \badcoupling(G)$.
Our strategy now is to impose additional
topological and algebraic conditions
so that the second term of \eqref{equ: bipartite facet split} also vanishes,
and then the tree solution parameterizes
a positive-dimensional solution for the entire system.
Of course, this strategy is unlikely going to
parameterize the entire positive-dimensional solution set.
But, as we will demonstrate, it is applicable for
a surprisingly large family of graphs.

\subsection{Cycles}

We first consider cycle networks.
In this context, \Cref{thm: unicycle non-bernshtein condition} can be paraphrased
as the following necessary condition for the existence of
positive-dimensional zero sets.

\begin{proposition}\label{thm: unicycle +dimensional necessary}
    Let $G = C_N$.
    If, for a choice of $K$ and $\colv{w}$,
    the $\CC^*$-zero sets of $\aks_{(G,K,\colv{w})}$
    is positive-dimensional, then $N$ is even and
    $(G,K)$ contains a balanced subnetwork
    (\Cref{def: weight-wise balancing}).
\end{proposition}

We now show that with the additional assumptions that
the oscillators are homogeneous
and the coupling coefficients on edges are among
$\{ \pm k \}$ for a $k \ne 0$,
the converse is also true.

\begin{proposition}[Non-isolated $\CC^*$-zero set for even cycles]\label{thm: unicycle +dimensional}
    Let $G=C_N$ for an even $N > 2$.
    Then there exists a choice of coupling coefficients $K$
    for which the algebraic Kuramoto system $\aks_{(G,K,w \colv{1})}$
    has infinitely many $\CC^*$-solutions.
    Specifically, for a fixed $k \in \CC^*$ any choice of
    $k_{ij} \in \{ \pm k \}$ with
    \begin{itemize}
        \item an even number of negative choices if $N \equiv 0 \mod 4$,
        \item an odd  number of negative choices if $N \equiv 2 \mod 4$,
    \end{itemize}
    the $\aks_{(G,K,\colv{0})}$
    derived from a homogeneous network
    has a positive-dimensional $\CC^*$-zero set,
    which contains the image of 
    $\rowv{x}_{\dig{T}}$ (\Cref{def: tree solution}),
    for $\dig{T} < G$.
    Moreover, if $k \in \RR \setminus \{0\}$,
    the corresponding transcendental Kuramoto system
    has a positive-dimensional real zero set.
\end{proposition}

Note that the conditions given here can be stated more succinctly:
It is equivalent to saying $G$ contains a balanced subnetwork
and the coupling coefficients are among $\{ \pm k \}$.

\begin{proof}
    Let $\dig{H}$ be an balanced subdigraph
    (\Cref{def: edge-wise balancing}) of $C_N$.
    Then the assumption ensures that $\dig{H}$
    is balanced with respect to $K$.

    Fix any weak spanning tree $\dig{T}$ of $\dig{H}$
    and let $\rowv{x}_{\dig{T}}$ be the induced
    tree solution defined in \Cref{def: tree solution}.
    With a straightforward calculation, we verify 
    \begin{equation*}
        \rinc(\dig{H}) \left(
            \rowv{x}_{\dig{T}}(\lambda)^{-\rinc(\dig{H})} \circ \rowv{a}(\dig{H})
        \right)^\top
        =
        \rinc(\dig{H}) \left(
            (\lambda \cdot \rowv{\eta} \circ \rowv{a}(\dig{H}))^{-I}
            \circ \rowv{a}(\dig{H})
        \right)^\top
        =
        \rinc(\dig{H}) \left(
            - \frac{ k^2 }{ 4 \lambda }
            \cdot
            \rowv{\eta}
        \right)^\top
        =
        \colv{0}.
    \end{equation*}
    Combined with the calculation above, we see that
    $\aks_G(\rowv{x}(\lambda)) = \colv{0}$ for all $\lambda \in \CC^*$
    and thus the $\CC^*$-zero set of $\aks_G = \colv{0}$ is positive-dimensional.

    Finally, by restricting $\lambda \in \CC^*$ to the unit circle $S^1$,
    and $k \in \mathbb{R}$,
    the $\CC^*$-orbit constructed here produces a
    1-dimensional real zero set to the original transcendental Kuramoto system
    \eqref{equ: kuramoto sin}.
\end{proof}

\begin{example}[4-cycle, again]\label{ex: C4 +dimensional}
    For $G = C_4$, as noted in \Cref{ex: C4},
    there can be as many as 6 balanced subnetworks,
    depending on the choices of coupling coefficients.
    They come in three pairs corresponding to reversing the orientation of all arcs.
    A representative of each pair is shown in \Cref{fig: C4 classes of facet subgraphs}.
    Each pair of balanced subnetworks
    produces an one-dimensional $\CC^*$-zero set of $\aks_G$
    through the formula given in \Cref{thm: unicycle +dimensional}.
    
    \begin{enumerate}[topsep=0pt,itemsep=0.75ex,partopsep=0ex,parsep=0ex]
    \item 
    Consider the first subdigraph $\dig{H}_1$ in \Cref{fig: C4 classes of facet subgraphs}.
    It is balanced if $k_{10} k_{32} = k_{12} k_{30}$.
    Then with the choice of $\dig{T}_1$ having arcs $(1,0), (1,2), (3,0)$,
    we can compute
    \begin{align*}
        \rinc(\dig{T}_1) &=
        \left[
        \begin{smallmatrix}
            +1 & +1 &  0 \\
             0 & -1 &  0 \\
             0 &  0 & +1
        \end{smallmatrix}
        \right]
        &
        \rinc(\dig{T}_1)^{-1} &=
        \left[
        \begin{smallmatrix}
            +1 & +1 &  0 \\
             0 & -1 &  0 \\
             0 &  0 & +1
        \end{smallmatrix}
        \right]
        &
        \rowv{\eta}_{\dig{T}_1} &=
        \begin{bmatrix}
            -1 & +1 & +1
        \end{bmatrix}
        ,
    \end{align*}
    where the notation is as in \Cref{thm: unicycle +dimensional}.
    Then the function $\rowv{x} : \CC^* \to (\CC^*)^3$, given by
    \begin{align*}
        \rowv{x}(\lambda) &=
        (
            \lambda \cdot \rowv{\eta}_{\dig{T}_1} \circ 
            \begin{bmatrix}
                a_{10} & a_{12} & a_{30}
            \end{bmatrix}^{-I}
        )^{ \rinc(\dig{T}_1)^{-1} }
        =
        \begin{bmatrix}
            -\frac{2\imag \lambda}{k_{10}} & 
            -\frac{k_{12}}{k_{10}} &
            \frac{2\imag \lambda}{k_{30}}
        \end{bmatrix}
    \end{align*}
    has image inside the $\CC^*$-zero set of $\aks_{C_4}$.
    In other words, the function
    \[
        \rowv{x}(\lambda)
        =
        \begin{bmatrix}
             -2\imag \lambda / k_{10} &
             -k_{12} / k_{10} &
             +2\imag \lambda / k_{30}
        \end{bmatrix}
        \quad\quad\text{if}\quad
        k_{10} k_{32} = k_{12} k_{30}
    \]
    parameterizes the open part of an one-dimensional orbit
    in the $\CC^*$-zero set of $\aks_{C_4}$.
    In the special case when $k_{ij} = 1$,
    this one-dimensional orbit is \emph{balanced}
    in the sense of \cite[Definition 3.1]{sclosa2022kuramoto}
    and \cite{sclosa2022kuramoto} provides analysis of the stability and geometry of such orbits.
    
    \item
    The second subnetwork $\dig{H}_2$
    in \Cref{fig: C4 classes of facet subgraphs}
    is balanced if $k_{10} k_{30} = k_{12} k_{23}$.
    With $\dig{T}_2 < \dig{H}_2$ given by arcs
    $(1,0), (1,2), (0,3)$ and  by following the construction above,
    we have 
    \[
        \rowv{x}(\lambda)
        =
        \begin{bmatrix}
             + 2\imag \lambda / k_{10} &
             - k_{12} / k_{10} &
             + k_{03} / 2\imag \lambda
        \end{bmatrix}
        \quad\quad\text{if}\quad
        k_{10} k_{30} = k_{12} k_{23}
    \]
    whose image is in the $\CC^*$-zero set of $\aks_{C_4}$
    for any $\lambda \in \CC^*$.
    Note that even in the special case of $k_{ij} = 1$,
    this orbit is \emph{not} balanced
    (in the sense of \cite[Definition 3.1]{sclosa2022kuramoto})
    for $\lambda \ne 1$.
    
    \item
    Finally, the third subnetwork $\dig{H}_3$
    in \Cref{fig: C4 classes of facet subgraphs}
    is balanced if $k_{10} k_{12} = k_{32} k_{03}$.
    With the choice $\dig{T}_3 < \dig{H}_3$,
    and following the construction above,
    we have 
    \[
        \rowv{x}(\lambda)
        =
        \begin{bmatrix}
             + k_{01} / 2\imag \lambda &
             - k_{01} k_{12} / 4 \lambda^2 &
             - k_{03} / 2\imag \lambda
        \end{bmatrix}
        \quad\text{if}\quad
        k_{10} k_{12} = k_{32} k_{30}
    \]
    whose image is in the $\CC^*$-zero set of $\aks_{C_4}$
    for any $\lambda \in \CC^*$.
    This orbit is more interesting as all three complex phase variables
    $x_1,x_2,x_3$ are nonconstant relative to the reference phase $x_0 = 1$.
    It is also \emph{not} balanced (in the sense of \cite[Definition 3.1]{sclosa2022kuramoto}).
    \end{enumerate}
    
    Moreover, if the coupling coefficients satisfy the condition
    $\rowv{k}^2(C_4) = c^2 \cdot \rowv{1}$ for some $c \in \RR$,
    then each of the three $\CC^*$-orbits contains 
    a one-real-dimensional components in the real torus $(S^1)^3$.
    Indeed, by the restriction $\lambda = \frac{c e^{\imag t}}{2 \imag}$,
    the three orbits constructed above reduce to parameterized zero sets
    of one real-dimension inside the real torus.
    If we define
    \[
        \sigma^{ij}_{\ell m} =
        \begin{cases}
            0 &\text{if } k_{ij}/k_{\ell m} > 0 \\
            \pi &\text{otherwise},
        \end{cases}
        \quad\text{and}\quad
        \sigma^{ij} =
        \begin{cases}
            0 &\text{if } k_{ij} > 0 \\
            \pi &\text{otherwise},
        \end{cases}
    \]
    then the three potential real orbits can be expressed as
    {\small
    \begin{align*}
        &
        \left\{
        \begin{aligned}
            \theta_1 &= t + \pi + \sigma^{10} \\
            \theta_2 &= \pi + \sigma^{12}_{10} \\
            \theta_3 &= t + \sigma^{30} \\
        \end{aligned}
        \right.
        \text{if } \ \frac{k_{10} k_{32}}{k_{12} k_{30}} = 1,
        &&
        \left\{
        \begin{aligned}
            \theta_1 &= \sigma^{10} + t \\
            \theta_2 &= \pi + \sigma^{21}_{10} \\
            \theta_3 &= \sigma^{30} - t
        \end{aligned}
        \right.
        \text{ if } \
        \frac{k_{10} k_{30}}{ k_{12} k_{23} } = 1,
        &&
        \left\{
        \begin{aligned}
            \theta_1 &= \sigma^{10} - t \\
            \theta_2 &= \sigma^{21}_{10} -2t \\
            \theta_3 &= \pi + \sigma^{30} -t \\
        \end{aligned}
        \right.
        \text{if } \
        \frac{ k_{10} k_{12} }{ k_{32} k_{30} } = 1,
    \end{align*}
    }%
    with $\theta_0 = 0$.
    It is easy to see if $k_{ij}$'s are identical,
    then all three real orbits exist.
    Moreover, these three orbits intersects at a singular point
    $(\theta_0,\theta_1,\theta_2,\theta_3) = (0, \pi/2, \pi, \pi/2)$.
    These positive-dimensional solution sets have been studied in~\cite{harrington2023kuramoto, LindbergZachariahBostonLesieutre2022Distribution,sclosa2022kuramoto}.
    In particular, \cite[Section 5.2]{sclosa2022kuramoto} provides a topological analysis for the orbits.
    In \Cref{ex: C4 +dimensional}, we showed they can also be derived systematically
    from balanced subnetworks.
\end{example}

\subsection{Multiple even cycles sharing one edge}

We now show that $\aks_G$ can have positive-dimensional
zero sets if $G$ is a graph consisting of
multiple even cycles sharing a single edge
(e.g., \Cref{fig: 8 again}).

\begin{proposition}
    Suppose $G$ consists of $d$ independent even cycles
    $C_1,\ldots,C_d$ that share a single edge $e$,
    then with the choice of the coupling coefficients
    \[
        k_{ij} =
        \begin{cases}
            sd &\text{if } \{i,j\} = e \\
            s  &\text{otherwise},
        \end{cases}
    \]
    for any $s \in \CC^*$,
    and homogeneous natural frequencies $\colv{w} = w \cdot \colv{1}$
    for any $w \in \CC^*$,
    the algebraic Kuramoto system $\aks_G$
    derived from the network $(G,K,w \cdot \colv{1})$
    has a positive-dimensional $\CC^*$-zero set.
\end{proposition}
\begin{proof}
    Since $G$ is bipartite, by \Cref{thm: faces are max bipartite},
    there exists a facet $F$ of $\radjp_G$ such that $G_F = G$,
    and thus $\radjp_G = \{ \colv{0} \} \cup F \cup (-F)$.
    Consequently,
    \[
        \aks_G(\rowv{x}) =
        \rinc(\dig{G}_F   )( \rowv{x}(\lambda)^{\rinc(\dig{G}_F   )} \circ \rowv{a}(\dig{G}_F)    )^\top +
        \rinc(\dig{G}_{-F})( \rowv{x}(\lambda)^{\rinc(\dig{G}_{-F})} \circ \rowv{a}(\dig{G}_{-F}) )^\top.
    \]
    Let $T$ be a spanning tree of $G$ that contains $e$,
    and let $\dig{T}$ be the corresponding subdigraph such that
    $\dig{T} < \dig{G}_F$.
    Consider the tree solution
        (\Cref{def: tree solution})
        \[
            \rowv{x}_{\dig{T}}(\lambda_1,\ldots,\lambda_d) =
            \left[
                \left( \sum_{k=1}^{d} \lambda_k \rowv{\tau}_k \right) \circ
                \rowv{a}_{\dig{T}}^{-I}
            \right]^{
                \rinc(\dig{T})^{-1}
            }
        \]
        where each vector $\rowv{\tau}_k$ contains coefficients
        with which an edge in $\edges(\dig{H}) \setminus \edges(\dig{T})$
        can be written as a linear combination of edges in $\dig{T}$.
        To prove the existence of positive-dimensional components,
        it is sufficient to restrict our attention to
        the case $\lambda = \lambda_1 = \ldots = \lambda_d$,
        i.e., the nonconstant function
        \[
            \rowv{x}(\lambda) := 
            \rowv{x}_{\dig{T}}(\lambda,\ldots,\lambda) =
            \left[
                \lambda \left( \sum_{k=1}^{d} \rowv{\tau}_k \right) \circ
                \rowv{a}_{\dig{T}}^{-I}
            \right]^{
                \rinc(\dig{T})^{-1}
            } =
            \left[
                \lambda \cdot \rowv{\eta}_{\dig{T}} \circ
                \rowv{a}_{\dig{T}}^{-I}
            \right]^{
                \rinc(\dig{T})^{-1}
            }
        \]
    in which the first $n-1$ entries of $\rowv{\eta}_{\dig{T}}$ are $\pm 1$,
    and its $n$-th entry is $\pm d$.
    By \Cref{prop: shared edge},
    \[
        \rinc(\dig{G}_F)( \rowv{x}(\lambda)^{\rinc(\dig{G}_F)} \circ \rowv{a}(\dig{G}_F) )^\top = \colv{0}
        \quad\text{for any } \lambda \in \CC^*.
    \]
    It remains to show 
    $\rinc(\dig{G}_{-F})( \rowv{x}(\lambda)^{\rinc(\dig{G}_{-F})} \circ \rowv{a}(\dig{G}_{-F}) )^\top = \colv{0}$.
    Since $\rowv{a}(\dig{T}) = \frac{\rowv{k}(\dig T)}{2\imag} = [\;s \; \cdots \; s \; sd\;]$,
    \begin{align*}
        \rowv{x}(\lambda)^{ - \rinc(\dig T) } \circ \rowv{a}(\dig T)
        &=
        (
            \lambda \cdot
            \rowv{\eta}_{\dig{T}} \circ
            \rowv{a}(\dig{T})^{-I}
        )^{-I}
        \circ \rowv{a}(\dig T)
        = 
        \lambda^{-1} 
        \rowv{\eta}_{\dig{T}}^{-I} \circ
        \begin{bmatrix}
            \frac{s^2}{-4} & \cdots & \frac{s^2}{-4} & \frac{s^2 d^2}{-4}
        \end{bmatrix}
        = 
        \frac{s^2}{-4\lambda} \,
        \rowv{\eta}_{\dig{T}}
    \end{align*}
    Moreover, by construction,
    $\rowv{x}(\lambda)^{ \colv{v}_i } \cdot \rowv{a}(e_i) = \lambda u_i$ 
    for $i=1,\ldots,d$.
    Therefore,
    \begin{align*}
        \rowv{x}(\lambda)^{ - \colv{v}_i } \cdot \rowv{a}(e_i)
        &= 
        \lambda^{-1} \rowv{a}^2(e_i) u_i^{-1} 
        =
        \frac{s^2}{-4 \lambda} \, u_i,
    \end{align*}
    since $u_1 \in \{ \pm 1 \}$.
    This shows that
    \begin{align*}
        \rowv{x}(\lambda)^{ \rinc(\dig{G}_{-F}) } \circ \rowv{a}(\dig{G}_{-F})
        &=
        \begin{bmatrix}
            \rowv{x}(\lambda)^{-\rinc(\dig T) } \circ \rowv{a}(\dig T) &
            \rowv{x}(\lambda)^{ - \colv{v}_1 } \cdot \rowv{a}(e_1) &
            \rowv{x}(\lambda)^{ - \colv{v}_2 } \cdot \rowv{a}(e_1)
        \end{bmatrix}
        =
        \frac{s^2}{-4 \lambda} \, \rowv{\eta}.
    \end{align*}
    Consequently, $\aks_G(\rowv{x}(\lambda)) = \colv{0}$ for any $\lambda \in \CC^*$,
    i.e., the $\CC^*$-zero set of $\aks_G$ is positive-dimensional.
\end{proof}

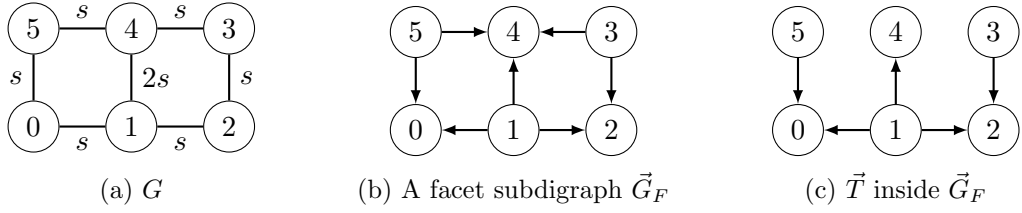
\begin{figure}
    \begin{subfigure}{0.3\textwidth}
        \centering
        \begin{tikzpicture}[
            scale=1.3,
            main/.style = {draw,circle,minimum size=1.5ex},
            every edge/.style = {draw,thick}
        ] 
            \node[main] (0) at ( 0, 0) {$0$}; 
            \node[main] (1) at ( 1, 0) {$1$}; 
            \node[main] (2) at ( 2, 0) {$2$}; 
            \node[main] (3) at ( 2, 1) {$3$}; 
            \node[main] (4) at ( 1, 1) {$4$}; 
            \node[main] (5) at ( 0, 1) {$5$}; 
            \path (0) edge node[below] {$s$} (1);
            \path (1) edge node[below] {$s$} (2);
            \path (2) edge node[right] {$s$} (3);
            \path (3) edge node[above] {$s$} (4);
            \path (4) edge node[above] {$s$} (5);
            \path (5) edge node[left]  {$s$} (0);
            \path (1) edge node[right] {$2s$} (4);
        \end{tikzpicture} 
        \caption{$G$}
        \label{fig: 8 again}
    \end{subfigure}~%
    \begin{subfigure}{0.3\textwidth}
        \centering
        \begin{tikzpicture}[
            scale=1.3,
            main/.style = {draw,circle,minimum size=1.5ex},
            every edge/.style = {draw,thick,-latex}
        ] 
            \node[main] (0) at ( 0, 0) {$0$}; 
            \node[main] (1) at ( 1, 0) {$1$}; 
            \node[main] (2) at ( 2, 0) {$2$}; 
            \node[main] (3) at ( 2, 1) {$3$}; 
            \node[main] (4) at ( 1, 1) {$4$}; 
            \node[main] (5) at ( 0, 1) {$5$}; 
            \path (1) edge (0);
            \path (1) edge (2);
            \path (3) edge (2);
            \path (3) edge (4);
            \path (5) edge (4);
            \path (5) edge (0);
            \path (1) edge (4);
        \end{tikzpicture} 
        \caption{A facet subdigraph $\dig{G}_F$}
        \label{fig: 8 facet again}
    \end{subfigure}~%
    \begin{subfigure}{0.3\textwidth}
        \centering
        \begin{tikzpicture}[
            scale=1.3,
            main/.style = {draw,circle,minimum size=1.5ex},
            every edge/.style = {draw,thick,-latex}
        ] 
            \node[main] (0) at ( 0, 0) {$0$}; 
            \node[main] (1) at ( 1, 0) {$1$}; 
            \node[main] (2) at ( 2, 0) {$2$}; 
            \node[main] (3) at ( 2, 1) {$3$}; 
            \node[main] (4) at ( 1, 1) {$4$}; 
            \node[main] (5) at ( 0, 1) {$5$}; 
            \path (1) edge (0);
            \path (1) edge (2);
            \path (3) edge (2);
            \path (5) edge (0);
            \path (1) edge (4);
        \end{tikzpicture} 
        \caption{$\dig{T}$ inside $\dig{G}_F$}
        \label{fig: 8 T}
    \end{subfigure}
    \caption{
        A network with two independent even cycles.
    }
\end{figure}
\begin{example}
    Consider the network shown in \Cref{fig: 8 again}.
    We fix the facet
    \[
        F = \conv \left\{
            \colv{e}_1,
            \colv{e}_1 - \colv{e}_2,
            \colv{e}_1 - \colv{e}_4,
            \colv{e}_3 - \colv{e}_2,
            \colv{e}_3 - \colv{e}_4,
            \colv{e}_5 - \colv{e}_4,
            \colv{e}_5 
        \right\}
        \subset \RR^5
    \]
    whose facet subdigraph is shown in \Cref{fig: 8 facet again}.
    We also fix a choice of $\dig{T} < \dig{G}_F$ shown in \Cref{fig: 8 T}.
    With this choice and the ordering of the arcs
    $(1,0)$, $(1,2)$, $(3,2)$, $(5,0)$, $(1,4)$, $(3,4)$, $(5,4)$,
    we have
    \begin{align*}
        \rinc(\dig T) &=
        \left[
        \begin{smallmatrix}
             1 &  1 & 0 & 0 & 1 \\
             0 & -1 &-1 & 0 & 0 \\
             0 &  0 & 1 & 0 & 0 \\
             0 &  0 & 0 & 0 &-1 \\
             0 &  0 & 0 & 1 & 0   
        \end{smallmatrix}
        \right]
        ,
        &
        \rinc(\dig T)^{-1} &=
        \left[
        \begin{smallmatrix}
             1 &  1 & 1 & 1 & 0 \\
             0 & -1 &-1 & 0 & 0 \\
             0 &  0 & 1 & 0 & 0 \\
             0 &  0 & 0 & 0 & 1 \\
             0 &  0 & 0 &-1 & 0   
        \end{smallmatrix}
        \right]
        ,
        &
        \colv{\eta}_1 &=
        \left[
        \begin{smallmatrix}
             0 \\
            -1 \\
            +1 \\
             0 \\
            +1 \\
            -1 \\
             0
        \end{smallmatrix}
        \right]
        ,
        &
        \colv{\eta}_2 &=
        \left[
        \begin{smallmatrix}
            -1 \\
             0 \\
             0 \\
            +1 \\
            +1 \\
             0 \\
            -1
        \end{smallmatrix}
        \right].
    \end{align*}
    We also define:
    \begin{align*}
        \rowv{\eta} &= \colv{\eta}_1^\top + \colv{\eta}_2^\top =
        \begin{bmatrix}
            -1 &
            -1 &
            +1 &
            +1 &
            +2 &
            -1 &
            -1
        \end{bmatrix}
        &&\text{and} &
        \rowv{\eta}_T &=
        \begin{bmatrix}
            -1 &
            -1 &
            +1 &
            +1 &
            +2
        \end{bmatrix}.
    \end{align*}
    We then verify that the construction
    $
        \rowv{x}(\lambda) = 
        (
            \lambda \cdot
            \rowv{\eta}_T
            \circ
            \rowv{a}(\dig T)^{-I}
        )^{ \rinc(\dig T)^{-1} }
    $
    above produces
    \begin{align*}
        &
        (
            \lambda
            \begin{bmatrix}
                -1 &
                -1 &
                +1 &
                +1 &
                +2
            \end{bmatrix}
            \circ
            \begin{bmatrix}
                \frac{s}{2\imag} & \frac{s}{2\imag} & \frac{s}{2\imag} & \frac{s}{2\imag} & \frac{2s}{2\imag}
            \end{bmatrix}^{-I}
        )^{
            \rinc(\dig T)^{-1}
        }
        &=
        \begin{bmatrix}
            - \frac{2\imag \lambda}{s} &
            +1 & 
            + \frac{2\imag \lambda}{s} &
            -1 &
            + \frac{2\imag \lambda}{s}
        \end{bmatrix}
        ,
    \end{align*}
    which
    parameterizes a one-dimensional $\CC^*$-zero of $\aks_G$.
    Moreover, by choosing $\lambda(t) = \frac{s e^{\imag t}}{2 \imag}$,
    $\rowv{x} \in (S^1)^5$ for any $t \in \RR$, $\log(\rowv{x})$ produces the one-dimensional real zero set
    \[
        (\theta_0,\theta_1,\theta_2,\theta_3,\theta_4,\theta_5) =
        (
            0, 
            t + \pi,
            0,
            t,
            \pi,
            t
        )
    \]
    for the transcendental Kuramoto system \eqref{equ: kuramoto sin}
    derived from the network in \Cref{fig: 8 again}.
\end{example}

\subsection{Homogeneous networks with uniform coupling coefficients}

We now consider the special case of networks 
with uniform coupling coefficients ($k_{ij}$ are all identical).
Without loss of generality, we fix $k_{ij} = 1$
and use the notation $(G,\rowv{1},\colv{0})$
to denote such a homogeneous and uniform network.
In this case, the conditions for
the exceptionality of coupling coefficients and
the existence of positive-dimensional zero sets
become topological constraints.
We conclude this section by listing a few families of graphs
on which a homogeneous Kuramoto network with uniform coupling coefficients
has a positive-dimensional set of synchronization configurations.

$G = C_{4m}$ for an integer $m \ge 1$
is a sufficient condition for having
a positive-dimensional zero set.
Its existence was first established in \cite{LindbergZachariahBostonLesieutre2022Distribution}. 
From the framework developed in this paper,
we can also construct such a 1-dimensional curve in the solution set 
using \Cref{thm: unicycle +dimensional},
since $(C_{4m},\rowv{1},\colv{0})$ is a balanced cycle.
Interestingly, this particular positive-dimensional zero set,
derived from $C_{4m}$,
can be extended directly to a large family of networks.

\begin{proposition}\label{thm: extending C4m}
    Suppose $G$ contains a cycle $C = C_{4m}$ for an $m \ge 2$.
    Let $V^+ \cup V^- = \nodes(C)$ be the partition of $C$
    as a bipartite graph.
    If any path in $G$ connecting nodes in $V^+$ and $V^-$
    includes at least one edge in $C$ then
    $\aks_{(G,\rowv{1}, \colv{0})}$ has a positive-dimensional $\CC^*$-zero set.
    Moreover, the corresponding transcendental Kuramoto system
    has a positive-dimensional real zero set.
\end{proposition}

\begin{proof}
    Without loss of generality,
    we can assume that node 0 is in $C$
    and $0 \in V^-$.
    According to \Cref{thm: unicycle +dimensional},
    there is a 1-dimensional solution to the subsystem
    $\aks_{(C,\rowv{1},\colv{0})}$ given by
    $$
    	x_i(\lambda) = \begin{cases}
    		\pm \lambda &\text{if } i \in V^+ \\
    		\pm 1.      &\text{if } i \in V^-
    	\end{cases}
    $$
    where the signs depends the choice of the weak spanning tree of $C$
    from which this is derived.

    By assumption,
    each node outside $C$ is either connected to nodes in $V^+$ or $V^-$
    through edges outside $C$.
    We label these two sets of nodes as $U^+$ and $U^-$ respectively,
    and we define
    $$
    	x_i(\lambda) = \begin{cases}
    		\lambda &\text{if } i \in U^+ \\
    		1.      &\text{if } i \in U^-
    	\end{cases}
    $$
    Then in both cases of $i,j \in U^+ \cup V^+$ and
    or $i,j \in U^- \cup V^-$, we have
    $$
    	\frac{x_i}{x_j} - \frac{x_j}{x_i} =
    	(\pm 1) - (\pm 1) = 0.
    $$
    Recall that polynomials $\aks_{(G,\rowv{1},\colv{0})}$
    are sums or differences between
    polynomials in $\aks_{(C,\rowv{1},\colv{0})}$ and
    $\frac{x_i}{x_j} - \frac{x_j}{x_i}$
    for edges $\{i,j\}$ in $\edges (G) \backslash \edges (C)$,
    which, by assumption, must be incident to nodes only
    in $U^+ \cup V^+$ or $U^-\cup V^-$.
    Therefore, 
    $(x_1(\lambda),\ldots, x_n(\lambda))$ defines a curve
    in the $\CC^*$-zero set of $\aks_{(G,\rowv{1},\colv{0})}$. 

    Finally, by restricting $\lambda$ to the unit circle $S^1$,
    i.e., setting $\lambda = e^{\imag \theta}$ for $\theta \in \RR$,
    we get a curve of in the real zero set of the
    transcendental Kuramoto system.
\end{proof}

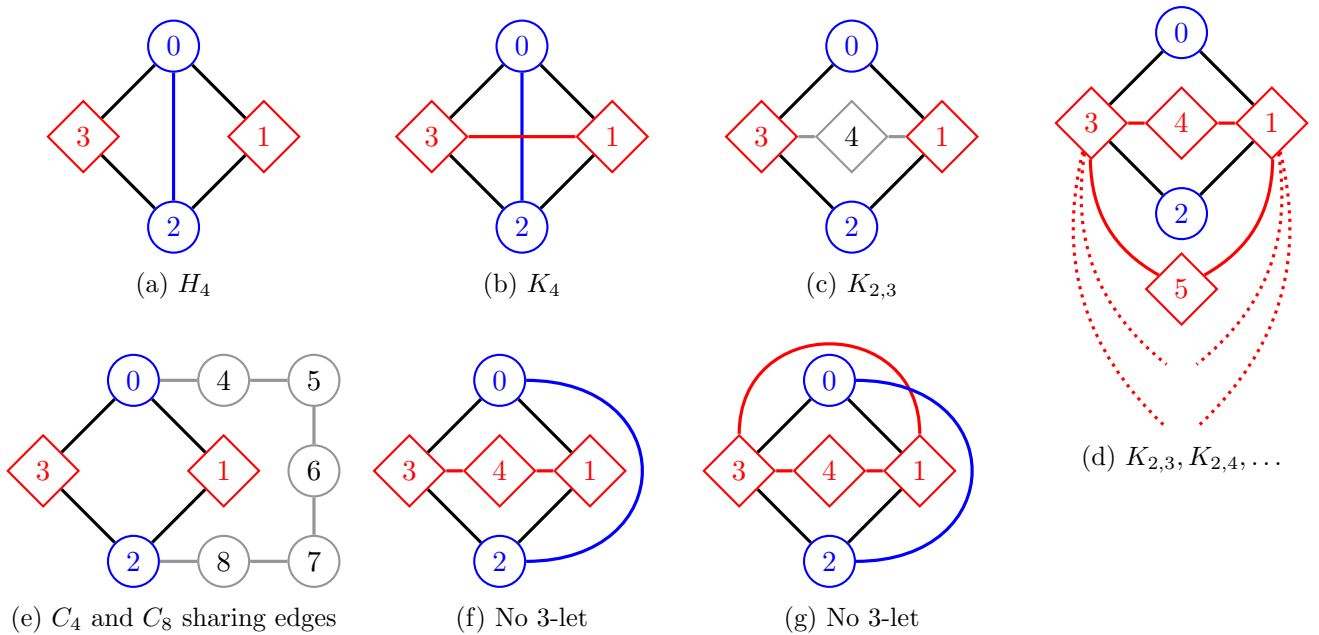
\begin{figure}[t]
    \begin{tabular}{cccc}
        \begin{subfigure}[b]{0.24\textwidth}
            \centering
            \begin{tikzpicture}[
                Negative/.style = {draw=blue,text=blue,thick,circle},
                Positive/.style = {draw=red,text=red,thick,diamond},
                every edge/.style = {draw,very thick},
                scale=0.80
            ] 
                \node[Negative] (0) at ( 0, 1.5) {$0$}; 
                \node[Negative] (2) at ( 0,-1.5) {$2$}; 
                \node[Positive] (1) at ( 1.5, 0) {$1$}; 
                \node[Positive] (3) at (-1.5, 0) {$3$}; 
                \path (1) edge (0);
                \path (1) edge (2);
                \path (3) edge (0);
                \path (2) edge (3);
                \path (0) edge[blue] (2);
            \end{tikzpicture} 
            \caption{$H_4$}
            \label{fig: H4 red blue}
        \end{subfigure}
        &
        \begin{subfigure}[b]{0.24\textwidth}
            \centering
            \begin{tikzpicture}[
                Negative/.style = {draw=blue,text=blue,thick,circle},
                Positive/.style = {draw=red,text=red,thick,diamond},
                every edge/.style = {draw,very thick},
                scale=0.80
            ] 
                \node[Negative] (0) at ( 0, 1.5) {$0$}; 
                \node[Negative] (2) at ( 0,-1.5) {$2$}; 
                \node[Positive] (1) at ( 1.5, 0) {$1$}; 
                \node[Positive] (3) at (-1.5, 0) {$3$}; 
                \path (1) edge (0);
                \path (1) edge (2);
                \path (3) edge (0);
                \path (2) edge (3);
                \path (0) edge[blue] (2);
                \path (1) edge[red] (3);
            \end{tikzpicture} 
            \caption{$K_4$}
            \label{fig: K4 red blue}
        \end{subfigure}
        &
        \begin{subfigure}[b]{0.24\textwidth}
            \centering
            \begin{tikzpicture}[
                Negative/.style = {draw=blue,text=blue,thick,circle},
                Positive/.style = {draw=red,text=red,thick,diamond},
                NegativeLike/.style = {draw=gray,text=black,thick,circle},
                PositiveLike/.style = {draw=gray,text=black,thick,diamond},
                every edge/.style = {draw,very thick},
                scale=0.80
            ] 
                \node[Negative] (0) at ( 0,1.5) {$0$}; 
                \node[Negative] (2) at ( 0,-1.5) {$2$}; 
                \node[Positive] (1) at ( 1.5, 0) {$1$}; 
                \node[Positive] (3) at (-1.5, 0) {$3$}; 
                \node[PositiveLike] (4) at ( 0,0) {$4$}; 
                \path (1) edge (0);
                \path (1) edge (2);
                \path (3) edge (0);
                \path (2) edge (3);
                \path (1) edge[gray] (4);
                \path (3) edge[gray] (4);
            \end{tikzpicture} 
            \caption{$K_{2,3}$}
            \label{fig: K23 red blue}
        \end{subfigure}
        &
        \multirow[t]{2}{*}[-10ex]{%
            \begin{subfigure}[t]{0.24\textwidth}
            \centering
            \begin{tikzpicture}[
                Negative/.style = {draw=blue,text=blue,thick,circle},
                Positive/.style = {draw=red,text=red,thick,diamond},
                every edge/.style = {draw,very thick},
                scale=0.8
            ] 
                \node[Negative] (0) at ( 0,1.5) {$0$}; 
                \node[Negative] (2) at ( 0,-1.5) {$2$}; 
                \node[Positive] (1) at ( 1.5, 0) {$1$}; 
                \node[Positive] (3) at (-1.5, 0) {$3$}; 
                \node[Positive] (4) at ( 0,0) {$4$}; 
                \node[Positive] (5) at ( 0,-2.75) {$5$}; 
                \path (1) edge (0);
                \path (1) edge (2);
                \path (3) edge (0);
                \path (2) edge (3);
                \path (1) edge[red] (4);
                \path (3) edge[red] (4);
                \path (1) edge[red,bend left] (5);
                \path (3) edge[red,bend right] (5);
                \path (1) edge[red,bend left,dotted]  ( 0.25,-4);
                \path (3) edge[red,bend right,dotted] (-0.25,-4);
                \path (1) edge[red,bend left,dotted]  ( 0.30,-5);
                \path (3) edge[red,bend right,dotted] (-0.30,-5);
            \end{tikzpicture} 
            \caption{$K_{2,3},K_{2,4},\ldots$}
            \label{fig: K2m}
        \end{subfigure}%
        }
        \\
        \begin{subfigure}[b]{0.27\textwidth}
            \centering
            \begin{tikzpicture}[
                Negative/.style = {draw=blue,text=blue,thick,circle},
                Positive/.style = {draw=red,text=red,thick,diamond},
                NegativeLike/.style = {draw=gray,text=black,thick,circle},
                PositiveLike/.style = {draw=gray,text=black,thick,diamond},
                every edge/.style = {draw,very thick},
                scale=0.80
            ] 
                \node[Negative] (0) at ( 0, 1.5) {$0$}; 
                \node[Negative] (2) at ( 0,-1.5) {$2$}; 
                \node[Positive] (1) at ( 1.5, 0) {$1$}; 
                \node[Positive] (3) at (-1.5, 0) {$3$}; 
                \node[NegativeLike] (4) at (1.5,1.5) {$4$};
                \node[NegativeLike] (5) at (3.0,1.5) {$5$};
                \node[NegativeLike] (6) at (3.0,0.0) {$6$};
                \node[NegativeLike] (7) at (3.0,-1.5) {$7$};
                \node[NegativeLike] (8) at (1.5,-1.5) {$8$};
                \path (1) edge (0);
                \path (1) edge (2);
                \path (3) edge (0);
                \path (3) edge (2);
                \path (0) edge[gray] (4);
                \path (4) edge[gray] (5);
                \path (5) edge[gray] (6);
                \path (6) edge[gray] (7);
                \path (7) edge[gray] (8);
                \path (8) edge[gray] (2);
            \end{tikzpicture} 
            \caption{$C_4$ and $C_8$ sharing edges}
            \label{fig: C4+C8 red blue}
        \end{subfigure}
        &
        \begin{subfigure}[b]{0.24\textwidth}
            \centering
            \begin{tikzpicture}[
                Negative/.style = {draw=blue,text=blue,thick,circle},
                Positive/.style = {draw=red,text=red,thick,diamond},
                every edge/.style = {draw,very thick},
                scale=0.80
            ] 
                \node[Negative] (0) at ( 0, 1.5) {$0$}; 
                \node[Negative] (2) at ( 0,-1.5) {$2$}; 
                \node[Positive] (1) at ( 1.5, 0) {$1$}; 
                \node[Positive] (3) at (-1.5, 0) {$3$}; 
                \node[Positive] (4) at ( 0,0) {$4$}; 
                \path (1) edge (0);
                \path (1) edge (2);
                \path (3) edge (0);
                \path (2) edge (3);
                \path (1) edge[red] (4);
                \path (3) edge[red] (4);
                \path (0) edge[blue,bend left=90,looseness=2.2] (2);
            \end{tikzpicture} 
            \caption{No 3-let}
            \label{fig: K23 relative 1}
        \end{subfigure}
        &
        \begin{subfigure}[b]{0.24\textwidth}
            \centering
            \begin{tikzpicture}[
                Negative/.style = {draw=blue,text=blue,thick,circle},
                Positive/.style = {draw=red,text=red,thick,diamond},
                every edge/.style = {draw,very thick},
                scale=0.80
            ] 
                \node[Negative] (0) at ( 0, 1.5) {$0$}; 
                \node[Negative] (2) at ( 0,-1.5) {$2$}; 
                \node[Positive] (1) at ( 1.5, 0) {$1$}; 
                \node[Positive] (3) at (-1.5, 0) {$3$}; 
                \node[Positive] (4) at ( 0,0) {$4$}; 
                \path (1) edge (0);
                \path (1) edge (2);
                \path (3) edge (0);
                \path (2) edge (3);
                \path (1) edge[red] (4);
                \path (3) edge[red] (4);
                \path (3) edge[red,bend left=90,looseness=1.7] (1);
                \path (0) edge[blue,bend left=90,looseness=2.2] (2);
            \end{tikzpicture} 
            \caption{No 3-let}
            \label{fig: K23 relative 2}
        \end{subfigure}
        &
    \end{tabular}
    \caption{
        Examples of homogeneous and uniform networks
        into which a solution curve derived from $C_4$ can be lifted
        according to \Cref{thm: extending C4m}.
        Nodes in $V^+, V^-, U^+, U^-$ are represented by
        red diamonds, blue circles 
        gray diamonds, and gray circles respectively.
    }
    \label{fig: C4 relatives}
\end{figure}

\begin{example}[Extending $C_4$ solutions]
    With this result, we can extend the 1-dimensional $\CC^*$-zero set
    for the Kuramoto system derived from
    a cycle network of 4 homogeneous oscillators
    with uniform coupling coefficients
    directly to other networks.
    \Cref{fig: C4 relatives}
    shows some of the examples of networks
    to which such extension is possible.
    This result gives us 1-dimensional curves
    \begin{align*}
        \rowv{x}_{\text{(a)}}(\lambda) &=
        \begin{bmatrix}
            -\lambda & -1 & \lambda
        \end{bmatrix}
        &
        \rowv{x}_{\text{(b)}}(\lambda) &=
        \begin{bmatrix}
            -\lambda & -1 & \lambda
        \end{bmatrix}
        \\
        \rowv{x}_{\text{(c)}}(\lambda) &=
        \begin{bmatrix}
            -\lambda & -1 & \lambda
            & \lambda
        \end{bmatrix}
        &
        \rowv{x}_{\text{(e)}}(\lambda) &=
        \begin{bmatrix}
            -\lambda & -1 & \lambda
            & 1 & 1 & 1 & 1 & 1
        \end{bmatrix}
    \end{align*}
    in their respective $\CC^*$-zero sets
    (as well as real zero sets via the restriction $\lambda = e^{\imag \theta}$).
    This shows that the respective system have
    positive-dimensional zero sets.
\end{example}

The graph $K_{2,3}$ in \Cref{fig: K23 red blue}
is a complete bipartite graph,
which consists of all edges between two sets of nodes
($\{1,3\}$ and $\{0,2,4\}$, in this case).
The construction in \Cref{thm: extending C4m}
extends trivially to $K_{2,m}$ for $m \ge 2$,
since $K_{2,m}$ can be created from $C_4$ by
adding nodes $4,\ldots,m+1$ and edges
$\{1,j\}$ and $\{3,j\}$ for $j=5,\ldots$,
as in \Cref{fig: K2m}.

\begin{corollary}\label{thm: K2m +dimensional}
    For any $m \ge 2$, $\aks_{(K_{2,m},\rowv{1}, \colv{0})}$
    has a positive-dimensional $\CC^*$-zero set,
    and the corresponding transcendental Kuramoto system
    has a positive-dimensional real zero set. 
\end{corollary}

\begin{example}[Beyond 3-lets]
    The example $K_{2,3}$ (\Cref{fig: K23 red blue})
    is noteworthy for another reason:
    It is an example of a graph that contains a \emph{3-let},
    i.e., a set of 3 nodes ($\{0,2,4\}$ in this case)
    sharing the same set of neighbors.
    In \cite[Corollary 2.9]{harrington2023kuramoto}
    Harrington, Schenck, and Stillman
    showed that the Kuramoto system derived from
    a homogeneous and uniform Kuramoto network
    based on a 2-connected graph that contains a 3-let
    must have a positive-dimensional zero set.
    \Cref{fig: K23 relative 1,fig: K23 relative 2}
    shows graphs obtained by
    adding edges to $K_{2,3}$, which no longer contain 3-lets.
    However, \Cref{thm: extending C4m} shows that
    there are still positive-dimensional $\CC^*$-zero sets.
    Indeed, the curve 
    \[
        \rowv{x}(\lambda) =
        \begin{bmatrix}
            -\lambda, -1, \lambda, \lambda
        \end{bmatrix}
    \]
    is contained in the $\CC^*$-zero sets of the
    Kuramoto systems derived from all three
    homogeneous and uniform networks
    despite the extra terms introduced by the added edges.
\end{example}

\begin{example}
    Similarly, starting with $C_8$,
    \Cref{thm: extending C4m} shows that the 1-dimensional curve
    \begin{equation}\label{equ: C8 solution}
        \rowv{x}(\lambda) = 
        \begin{bmatrix}
            -\lambda & -1 & \lambda & 1 & -\lambda & -1 & \lambda
        \end{bmatrix}
    \end{equation}
    can be embedded into the zero set
    of the Kuramoto system derived from
    the network in \Cref{fig: C8 + 1}.
    
    We end with the observation that
    for simplicity
    the conditions in \Cref{thm: extending C4m}
    is made too restrictive.
    For instance, even though the network in \Cref{fig: C8 + many edges}
    does not satisfy the condition in this proposition,
    the solution \eqref{equ: C8 solution}
    is still a valid curve solution
    to the derived Kuramoto system.
    This implies that 1-dimensional cycle solutions can be extended to a much broader class of networks.
\end{example}

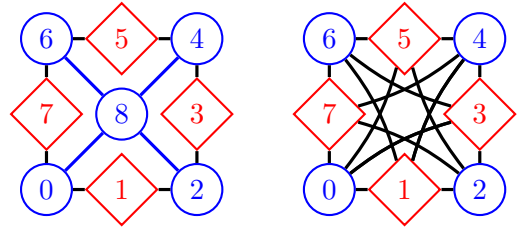
\begin{wrapfigure}{r}{0.45\textwidth}
    \begin{subfigure}[b]{0.22\textwidth}
        \centering
        \begin{tikzpicture}[
            Negative/.style = {draw=blue,text=blue,thick,circle},
            Positive/.style = {draw=red,text=red,thick,diamond},
            every edge/.style = {draw,very thick},
            scale=1.0
        ] 
            \node[Negative] (0) at ( 0, 0) {$0$}; 
            \node[Positive] (1) at ( 1, 0) {$1$}; 
            \node[Negative] (2) at ( 2, 0) {$2$}; 
            \node[Positive] (3) at ( 2, 1) {$3$}; 
            \node[Negative] (4) at ( 2, 2) {$4$}; 
            \node[Positive] (5) at ( 1, 2) {$5$}; 
            \node[Negative] (6) at ( 0, 2) {$6$}; 
            \node[Positive] (7) at ( 0, 1) {$7$}; 
            \node[Negative] (8) at ( 1, 1) {$8$}; 
            \path (0) edge (1);
            \path (1) edge (2);
            \path (2) edge (3);
            \path (3) edge (4);
            \path (4) edge (5);
            \path (5) edge (6);
            \path (6) edge (7);
            \path (7) edge (0);
            \path (8) edge[blue] (6);
            \path (8) edge[blue] (4);
            \path (8) edge[blue] (0);
            \path (8) edge[blue] (2);
        \end{tikzpicture} 
        \caption{Adding edges to $C_8$}
        \label{fig: C8 + 1}
    \end{subfigure}
    \begin{subfigure}[b]{0.22\textwidth}
        \centering
        \begin{tikzpicture}[
            Negative/.style = {draw=blue,text=blue,thick,circle},
            Positive/.style = {draw=red,text=red,thick,diamond},
            every edge/.style = {draw,very thick},
            scale=1.0
        ] 
            \node[Negative] (0) at ( 0, 0) {$0$}; 
            \node[Positive] (1) at ( 1, 0) {$1$}; 
            \node[Negative] (2) at ( 2, 0) {$2$}; 
            \node[Positive] (3) at ( 2, 1) {$3$}; 
            \node[Negative] (4) at ( 2, 2) {$4$}; 
            \node[Positive] (5) at ( 1, 2) {$5$}; 
            \node[Negative] (6) at ( 0, 2) {$6$}; 
            \node[Positive] (7) at ( 0, 1) {$7$}; 
            \path (0) edge (1);
            \path (1) edge (2);
            \path (2) edge (3);
            \path (3) edge (4);
            \path (4) edge (5);
            \path (5) edge (6);
            \path (6) edge (7);
            \path (7) edge (0);
            \path (0) edge[bend left=11]  (3);
            \path (0) edge[bend right=11] (5);
            \path (1) edge[bend left=11]  (4);
            \path (1) edge[bend right=11] (6);
            \path (2) edge[bend left=11]  (5);
            \path (2) edge[bend right=11] (7);
            \path (3) edge[bend left=11]  (6);
            \path (3) edge[bend right=11] (0);
            \path (4) edge[bend left=11]  (7);
            \path (4) edge[bend right=11] (1);
        \end{tikzpicture} 
        \caption{Adding more edges}
        \label{fig: C8 + many edges}
    \end{subfigure}
    \caption{
        Examples of homogeneous and uniform networks
        to which a cycle solution derived from $C_8$
        can be lifted into.
    }
    \label{fig: C8 relatives}
\end{wrapfigure}

\section{Concluding remarks}\label{sec: conclusion}

We studied the structure of the zero sets of the
Kuramoto equations and its algebraic counterpart.
By leveraging a recently discovered
three-way connection between the
graph-theoretic, convex-geometric, and tropical view points, 
we answered three key questions.

    For \Cref{q1},
    we showed that
    \emph{for generic natural frequencies and
        generic but symmetric coupling coefficients,
        the $\CC^*$-root count of the algebraic Kuramoto system
        coincides with the adjacency polytope bound,
        and as a corollary we showed that the algebraic Kuramoto system is Bernshtein-general.
    }
    The proof is constructive,
    and it produces a homotopy continuation method for computing all zeros.
    For \Cref{q2}, we provided a graph-theoretic description of the
    \emph{
        exceptional coupling coefficients
        for which the $\CC^*$-root count
     drops below the generic root count.}
    For \Cref{q3}, 
    we established sufficient conditions
    on many families of networks
    where \emph{there will be non-isolated real and complex zero sets} for these systems
    through explicit constructions.

While the analysis for the last two questions required some
topological restrictions,
it appears hopeful that
the approach taken here can be generalized to other networks.
We hope our work will spark interest in the
full analysis of the typical and atypical solutions to Kuramoto equations.

Finally, we speculate that the explicit connection between the root count
of the algebraic Kuramoto equations and the normalized volume of adjacency polytopes may also allow algebraic geometers
to directly contribute to the geometric study of adjacency polytopes.

\section*{Acknowledgments}

This project is inspired by a series of discussion
the first named author had with
Anton Leykin, Josephine Yu, and Yue Ren
between 2017 and 2018.
The first named author learned much about
the structure of adjacency polytopes (symmetric edge polytopes)
from Robert Davis,
Alessio D’Al\`i, Emanuele Delucchi, and Mateusz Micha\l{}ek.
The authors thank Paul Breiding, Paul Helminck, and Davide Sclosa
for their comments on an earlier version of this manuscript.

\appendix
\section{Notation}\label{app: notations}
Here, we list notation used in this paper
that may not be standard.

\begin{description}[itemsep=1.2ex]
\item[$\CC\{\tau\}$]
    The field of Puiseux series in $\tau$ with complex coefficients
    is denoted $\CC\{\tau\}$. Only convergent series
    representing germs of one-dimensional analytic varieties
    will be relevant.
    
\item[$\seccone(P,\Delta_\omega)$]
    The closed secondary cone of a regular subdivision $\Delta_\omega$
    of a polytope $P$.
    
\item[$\subdiv_{\omega}(P)$]
    The regular subdivision of a point configuration $P$
    induced by a lifting function $\omega : P \to \QQ$.
    
\item[$\radjp_G$]
    The point configuration associated with the adjacency polytope
    derived from a connected graph $G$.
    It is defined to be
    $\{ \pm ( \colv{e}_i - \colv{e}_j ) \mid \{i,j\} \in \edges(G) \} \cup \{ \colv{0} \}$.

\item[$\radjp_G^\omega$]
    The ``lifted'' version of the point configuration $\radjp_G$
    induced by a lifting function $\omega : \radjp_G \to \QQ$.
    It consists of the points $(\colv{p}, \omega(\colv{p})) \subset \RR^{n+1}$
    for all $\colv{p} \in \radjp_G$.
    
\item[$\aks_{(G,K,\colv{w})}, \aks_{(G,K)}, \aks_{G}$]
    The algebraic Kuramoto system derived from a Kuramoto network $(G,K,\colv{w})$.
    If the choice of the natural frequencies $\colv{w}$ is not relevant
    (or assumed to be generic),
    the notation $\aks_{(G,K)}$ is used.
    Similarly, if only graph topology is of relevance,
    we use $\aks_G$.
    
\item[$\raks_G$]
    The randomized algebraic Kuramoto system,
    i.e., $R \cdot \aks_G$ for a generic square matrix $R$.
    
\item[$(G,K,\colv{w}), (G,K)$]
    A Kuramoto network.
    $G$ is the underlying graph,
    $K = \{ k_{ij} \mid \{i,j\} \in \edges(G) \}$ with $k_{ij} = k_{ji}$
    encodes the coupling coefficients,
    and $\colv{w} = (w_0,\ldots,w_n)^\top$ contains the natural frequencies.

\item[$\init_{\rowv{v}}(f)$]
    For a Laurent polynomial $f$ in $x_1,\ldots,x_n$,
    the \emph{initial form} of $f$ with respect to a vector $\rowv{v}$
    is the Laurent polynomial
    $\init_{\rowv{v}}(f)(\rowv{x}) := \sum_{\colv{a} \in (S)_{\colv{v}}} c_{\colv{a}} \, \rowv{x}^{\colv{a}}$, 
    where $(S)_{\colv{v}}$ is the subset of $S$ on which 
    $\inner{ \rowv{v} }{ \cdot }$ is minimized.
    It is denoted $\init_{\rowv{v}}(f)$.
    For a system $\colv{f} = (f_1,\ldots,f_q)$ of Laurent polynomials,
    $\init_{\rowv{v}}(\colv{f}) = (\init_{\rowv{v}} (f_1),\ldots,\init_{\rowv{v}} (f_q))$.
    
\item[$Q(\dig{H}),\rinc(\dig{H})$]
For a digraph $\dig{H}$, its \emph{incidence matrix} $Q(\dig{H})$
is the matrix with columns $\colv{e}_i - \colv{e}_j$
such that $(i,j) \in \edges(\dig{H})$.
Since we set $\colv{e}_0 $ to be the zero vector,
the first row is all zeros.
Therefore, we instead consider the \emph{reduced incidence matrix},
$\rinc(\dig H)$, with $n = |\nodes(\dig H)| - 1$ rows,
which is the incidence matrix of $\dig{H}$ with the first row deleted.
    The ordering of the columns in both is arbitrary,
    but, when either matrix appears in the same context with other incidence vectors,
    a consistent ordering is assumed.
    Here, the adjective ``reduced'' emphasizes the fact that
    the labels for the nodes in the graph are $0,1,\ldots$,
    and therefore, for a digraph $\dig{H}$ of $n+1$ nodes,
    $\rinc(\dig{H})$ only has $n$ rows.

\item[$\operatorname{MV}(P_1,\ldots,P_n)$]
    Given $n$ convex polytopes $P_1,\ldots,P_n \subset \RR^n$,
    the \emph{mixed volume} of $P_1,\ldots P_n$
    is the coefficient of the monomial $\lambda_1 \cdots \lambda_n$
    in the homogeneous polynomial
    $\operatorname{vol}_n(\lambda_1 P_1 + \ldots + \lambda_n P_n)$
    where $P + Q = \{ p + q \ : \ p \in P, \ q \in Q \}$
    denotes the Minkowski sum and $\operatorname{vol}_n$
    is the standard $n-$dimensional Euclidean volume.
    
\item[$\neighbors_G(i),\neighbors_{\dig G}^+(i),\neighbors_{\dig G}^-(i)$]
    For a graph $G$ and a node $i$ of $G$,
    $\neighbors_G(i)$ is the set of nodes that are adjacent to $i$.
    Similarly, the other two are the adjacent nodes
    through outgoing and incoming arcs in a digraph, respectively.
    
\item[$\nvol$]
    The normalized volume of a set in $\RR^n$,
    which is defined to be $n!$ times the Euclidean volume.
    The usage is restricted to convex polytopes here,
    and we adopt the convention that $\nvol(X) = 0$,
    if $X$ is not full dimensional.
\end{description}
    
\section{Elementary lemmata}\label{sec: lemmas}

We restate \Cref{lem: simplex tree}
and provide an elementary proof.

\simplextree*

\begin{proof}
    Let $(\colv{\alpha},1)$ be the upward pointing inner normal
    that defines the cell $\Delta$ as a projection of a lower facet of $\radjp_G^{w}$.
    Suppose $\dig{G}_\Delta$ has a simple directed cycle $i_1 \to \cdots \to i_m \to i_1$.
    Then
    \[
        \inner{ \colv{\alpha} }{ \colv{e}_{i_k} - \colv{e}_{i_{k+1}} } + 1 + \delta_{i_k,i_{k+1}} = 0
        \quad\text{for } k = 1,\ldots,m,
    \]
    where $i_{m+1} = i_1$.
    Summing these $m$ equations produces
    $
        m + \sum_{k=1}^m \delta_{i_k,i_{k+1}} = 0,
    $
    which is not possible under the assumption that $\delta_{ij}$
    are sufficiently close to 0.
    So $\dig{G}_\Delta$ must be acyclic.
    
    Moreover, $\dim(\Delta) = n$, by assumption.
    That is, $\{ \colv{e}_i - \colv{e}_j \mid (i,j) \in \dig{G}_\Delta \}$
    is must span $\RR^n$ as a set of vectors
    (since $\colv{0} \in \Delta$, by construction).
    Therefore, for every $i \in \{0,\ldots,n\}$,
    either $\colv{e}_i - \colv{e}_j \in \Delta$
    or     $\colv{e}_j - \colv{e}_i \in \Delta$
    for some $j \in \{0,\ldots,n\} \setminus \{i\}$.
    That is, $G_\Delta$ must be spanning.
\end{proof}

Now we restate \Cref{lem: componentwise facial system}
and provide a short proof.

\componentwise*

\begin{proof}
    Let $V_i = \nodes(\dig{H}_i)$ for $i=1,\ldots,\ell$ with $0 \in V_1$.
    We arrange the coordinates of $\ZZ^n$
    according to the grouping $V_1,\ldots,V_\ell$
    so that an inner normal vector $\rowv{\alpha} \in \RR^n$
    that defines the face $F$ of $\radjp_G$ is written as
    $
        \rowv{\alpha} = 
        \begin{bmatrix}\,
            \rowv{\alpha}_1 &
            \cdots &
            \rowv{\alpha}_\ell \,
        \end{bmatrix}
    $
    where $\rowv{\alpha}_i \in \RR^{|V_i|}$
    corresponds to nodes in $V_i$.
    Then for $i=1,\ldots,\ell$,
    \[
        \rowv{v}_i :=
        \begin{bmatrix}
            \,
            \rowv{0}_{ |V_1| + \cdots + |V_{i-1}| } &
            \rowv{\alpha}_i &
            \rowv{0}_{|V_{i+1}| + \cdots + |V_\ell|} \,
        \end{bmatrix}
        \in \RR^n
    \]
    defines a face $F_i$ of $\radjp_G$
    such that $\dig{H}_i = \dig{G}_{F_i}$.
    
    By grouping the rows and columns of $\rinc(\dig{G}_F)$
    according to the nodes and arcs in $H_1,\ldots,H_\ell$,
    $\rinc(\dig{G}_F)$ has the block structure
    \[
        \rinc(\dig{G}_F) =
        \left[
        \begin{smallmatrix}
            \rinc(\dig{H}_1) & &  \\
            & \ddots & \\
            & & \rinc(\dig{H}_\ell)
        \end{smallmatrix}
        \right].
    \]
    Therefore, the cycle form of $\init_F(\raks_G) = \colv{0}$,
    i.e., 
    $
        \rinc(\dig{G}_F)
        \left( 
            \rowv{x}^{ \rinc(\dig{G}_F) }
            \circ
            \rowv{a}(\dig{G}_F)
        \right)^\top
        =
        \colv{0},
    $
    is equivalent to
    \[
        \rinc(\dig{H}_i)
        \left( 
            \rowv{x}_i^{ \rinc(\dig{H}_i) }
            \circ
            \rowv{a}(\dig{H}_i)
        \right)^\top
        =
        \colv{0}
        \quad\text{for each } i = 1,\ldots,\ell,
    \]
    where $\rowv{x}_i$ contain the coordinates corresponding to nodes in $\dig{H}_i$.
    This is equivalent to
    $\init_{F_i}(\raks_G)(\rowv{x}) = \colv{0}$ for each $i=1,\ldots,\ell$,
    by \Cref{lem: face system}.
\end{proof}

Similarly, we restate \Cref{lem: smaller sub graph}
and provide a simple proof.

\smallergraph*

\begin{proof}
    Order the vertices of $\dig{G}_F$ so that 
    \[ \rinc(\dig{G}_F) = \begin{bmatrix}
        \rinc(\dig{G}_{F'}) \\ 0
    \end{bmatrix}, \]
    Write $\rowv{x} = [\rowv{x}' \ \rowv{y}]$ where
    $\rowv{x}'$ is the vector of variables appearing in
    $\init_{F'}(\raks_{G'})$, and
    let $e_1,\ldots,e_d$ be the arcs in $\dig{G}_F$
    that are not in $\dig{G}_{F'}$.
    Then by \Cref{lem: face system} we want to show that 
    \[ \rinc(\dig{G}_F)
        \left( 
            \rowv{x}^{ \rinc(\dig{G}_F) }
            \circ
            \rowv{a}(\dig{G}_F)
        \right)^\top
        =
        \colv{0} \quad \Leftrightarrow \quad \rinc(\dig{G}_{F'})
        \left( 
            \rowv{x'}^{ \rinc(\dig{G}_{F'}) }
            \circ
            \rowv{a}(\dig{G}_{F'})
        \right)^\top
        =
        \colv{0}. \]
        Observe that 
        \begin{align*}
            \rinc(\dig{G}_F)
        \left( 
            \rowv{x}^{ \rinc(\dig{G}_F) }
            \circ
            \rowv{a}(\dig{G}_F)
        \right)^\top &= \begin{bmatrix}
        \rinc(\dig{G}_{F'}) \\ 0
    \end{bmatrix} \left( 
            \rowv{[x' \ y]}^{ \rinc(\dig{G}_F) }
            \circ
            [\rowv{a}(\dig{G}_{F'}) \ a_{e_1} \ \ldots a_{e_d} ]
        \right)^\top \\
        &= \rinc(\dig{G}_{F'})
        \left( 
            \rowv{x'}^{ \rinc(\dig{G}_{F'}) }
            \circ
            \rowv{a}(\dig{G}_{F'})
        \right)^\top,
        \end{align*}
        proving the statement.
\end{proof}

We recall \Cref{cor: tree solution condition}
and provide an elementary proof.

\TreeCondition*

\begin{proof}
    We will repeat some of the steps in the proof of
    \Cref{lem: initial soln must be tree soln} for reference.
    
    Suppose we have $\lambda_1,\ldots,\lambda_d$
    for which the equations above hold.
    Then
    \[
        \left(
            \sum_{k=1}^d \lambda_k \rowv{\tau}_k
        \right)^{\colv{\tau}_i}
        \cdot
        \rowv{a}(\dig{T})^{-{\colv{\tau}_i}}
        \cdot \alpha_i
        =
        - \lambda_i
        \quad\text{for } i=1,\ldots,d.
    \]
    Let $\rowv{x}$ be the corresponding tree solution,
    i.e.,
    \[
        \rowv{x} =
        \rowv{x}_{\dig{T}}(\lambda_1,\ldots,\lambda_d) =
        \left[
            \left(
                \sum_{k=1}^d \lambda_k \rowv{\tau}_k
            \right)
            \circ \rowv{a}(\dig{T})^{-I}
        \right]^{\rinc(\dig{T})^{-1}}.
    \]
    Then
    \[
        \left(
            \rowv{x}^{ \rinc(\dig{H}) }
            \circ
            \rowv{a}(\dig{H})
        \right)^\top
        =
        \begin{bmatrix}
            \left( \rowv{x}^{\rinc(\dig{T}) } \circ \rowv{a}(\dig{T}) \right)^\top
            \\
            \rowv{x}^{\colv{v}_1} \cdot \alpha_1 \\
            \vdots \\
            \rowv{x}^{\colv{v}_d} \cdot \alpha_d \\
        \end{bmatrix}
        =
        \begin{bmatrix}
            \left(
                \sum_{k=1}^d \lambda_k \rowv{\tau}_k
            \right)^\top
            \\
            \left(
                \sum_{k=1}^d \lambda_k \rowv{\tau}_k
            \right)^{\rowv{\tau}_1}
            \rowv{a}(\dig{T})^{-\rowv{\tau}_1}
            \alpha_1
            \\
            \vdots
            \\
            \left(
                \sum_{k=1}^d \lambda_k \rowv{\tau}_k
            \right)^{\rowv{\tau}_d}
            \rowv{a}(\dig{T})^{-\rowv{\tau}_d}
            \alpha_d
        \end{bmatrix}
        =
        \begin{bmatrix}
            \left(
                \sum_{k=1}^d \lambda_k \rowv{\tau}_k
            \right)^\top
            \\
            -\lambda_1
            \\ \vdots \\
            -\lambda_d
        \end{bmatrix}
    \]
    which is in the null space of $\rinc(\dig{H})$.
    Thus,
    $\rinc(\dig{H}) (
        \rowv{x}^{\rinc(\dig{H})}
        \circ 
        \rowv{a}(\dig{H})
    )^\top = \colv{0}$.
    That is, $\rowv{x}$ is a $\CC^*$-solution
    to the facial system $\init_F(\raks_G)$.

    Conversely, suppose $\rowv{x}$ is a $\CC^*$-zero
    of the facial system $\init_F(\raks_G)$,
    then by \Cref{lem: initial soln must be tree soln},
    it is given by a tree solution. I.e.,
    $\rowv{x} = \rowv{x}_{\dig{T}}(\lambda_1,\ldots,\lambda_d)$
    for some $\lambda_1,\ldots,\lambda_d$.
    From the equation above, 
    \[
        \left(
            \rowv{x}^{ \rinc(\dig{H}) }
            \circ
            \rowv{a}(\dig{H})
        \right)^\top
        =
        \begin{bmatrix}
            \left(
                \sum_{k=1}^d \lambda_k \rowv{\tau}_k
            \right)^\top
            \\
            \left(
                \sum_{k=1}^d \lambda_k \rowv{\tau}_k
            \right)^{\rowv{\tau}_1}
            \rowv{a}(\dig{T})^{-\rowv{\tau}_1}
            \alpha_1
            \\
            \vdots
            \\
            \left(
                \sum_{k=1}^d \lambda_k \rowv{\tau}_k
            \right)^{\rowv{\tau}_d}
            \rowv{a}(\dig{T})^{-\rowv{\tau}_d}
            \alpha_d
        \end{bmatrix}
    \]
    must be in the null space of $\rinc(\dig{H})$.
    Since the vectors
    $
        (\rowv{\tau}_1,-\rowv{e}_1)^\top,\ldots,
        (\rowv{\tau}_d,-\rowv{e}_d)^\top
    $
    form a basis for this null space.
    Therefore,
    \[
        \left(
            \sum_{k=1}^d \lambda_k \rowv{\tau}_k
        \right)^{\colv{\tau}_i}
        \cdot
        \rowv{a}(\dig{T})^{-{\colv{\tau}_i}}
        \cdot \alpha_i
        =
        - \lambda_i
        \quad\text{for } i = 1,\ldots,d,
    \]
    which is equivalent to the conclusion.
\end{proof}

\bibliographystyle{siamplain}
\bibliography{ref}

\end{document}